\documentclass[a4paper,11pt]{article}
\usepackage{xargs,hyperref}
\usepackage[pdftex,dvipsnames]{xcolor}
\usepackage{comment,amsmath,amssymb,amsthm,changepage,pifont,graphicx,faktor,tikz}
\usepackage[english]{babel}
\usetikzlibrary{patterns}
\usepackage[colorinlistoftodos,prependcaption,textsize=tiny]{todonotes}
\newcommandx{\unsure}[2][1=]{\todo[linecolor=red,backgroundcolor=red!25,bordercolor=red,#1]{#2}}
\newcommandx{\info}[2][1=]{\todo[linecolor=OliveGreen,backgroundcolor=OliveGreen!25,bordercolor=OliveGreen,#1]{#2}}
\usepackage[font=small]{caption}

\hypersetup{
    colorlinks,
    linkcolor={red!50!black},
    citecolor={blue!50!black},
    urlcolor={blue!80!black}
}

\allowdisplaybreaks

\newtheorem{theorem}{Theorem}[section]

\newtheorem{lemma}[theorem]{Lemma}
\newtheorem{corollary}[theorem]{Corollary}

\newtheorem{conjecture}[theorem]{Conjecture}
\newtheorem{definition}[theorem]{Definition}

\theoremstyle{remark}
\newtheorem{remark}[theorem]{Remark}
\newtheorem{example}[theorem]{Example}

\DeclareMathOperator{\lw}{lw}
\DeclareMathOperator{\conv}{conv}

\DeclareMathOperator{\proj}{Proj}

\DeclareMathOperator{\vol}{vol}

\DeclareMathOperator{\GL}{GL}
\DeclareMathOperator{\AGL}{AGL}
\DeclareMathOperator{\codim}{codim}
\DeclareMathOperator{\coker}{coker}
\DeclareMathOperator{\im}{im}
\newcommand{\RR}{\mathbb{R}}
\newcommand{\ZZ}{\mathbb{Z}}
\newcommand{\PP}{\mathbb{P}}
\newcommand{\interior}{^{(1)}}
\newcommand{\topinterior}{^{\circ}}
\newcommand{\wedgepow}[1]{{\bigwedge}^{\!{#1}}\,}

\begin{document}

\title{Computing graded Betti tables of toric surfaces}
\author{W.\ Castryck, F.\ Cools, J.\ Demeyer, A.\ Lemmens}
\date{}
\maketitle

\begin{abstract}
We present various facts on the graded Betti table of a projectively embedded
toric surface, expressed in terms of the combinatorics of its defining lattice polygon. These facts include explicit formulas for a number of entries, as well as a lower bound on the length
of the linear strand that we conjecture to be sharp (and prove to be so in several special cases).
We also present an algorithm for determining the graded Betti table of a given toric surface by explicitly computing its Koszul cohomology, 
and report on an implementation in SageMath. 
It works well for ambient projective spaces of dimension
up to roughly $25$, depending on the concrete combinatorics, although the current implementation runs in finite characteristic only.
As a main application we obtain the graded Betti table of the Veronese surface $\nu_6(\PP^2) \subseteq \PP^{27}$ in characteristic $40\,009$. This allows us to formulate precise conjectures
predicting what certain entries look like in the case of an arbitrary Veronese surface $\nu_d(\PP^2)$.

\end{abstract}

\small
\tableofcontents

\normalsize

\section{Introduction}


Let $k$ be a field of characteristic $0$ and let $\Delta \subseteq \RR^2$ be a lattice polygon,
by which we mean the convex hull of
a finite number of points of the standard lattice $\ZZ^2$.
We write $\Delta^{(1)}$ for the convex hull of
the lattice points in the interior of $\Delta$. Assume that $\Delta$ is two-dimensional, write $N_\Delta = | \Delta \cap \ZZ^2 |$,
and let $S_\Delta = k[  X_{i,j}  \, | \,  (i,j) \in \Delta \cap \ZZ^2  ]$, so that $\PP^{N_\Delta-1} = \proj S_\Delta$.
The toric surface over $k$ associated to $\Delta$
is the Zariski closure of the image of
\[ \varphi_\Delta : (k^*)^2 \hookrightarrow \mathbb{P}^{ N_\Delta - 1} : (a, b) \mapsto (a^i b^j)_{(i,j) \in \Delta \cap \ZZ^2}. \]
We denote it by $X_\Delta$ and its ideal by $I_\Delta$. 
It has been proved by Koelman~\cite{koelman} that $I_\Delta$ is generated by binomials
of degree $2$ and $3$, where degree $2$ suffices if and only if $| \partial \Delta \cap \ZZ^2 | > 3$. 

Our object of interest is the graded Betti table of $X_\Delta$, which gathers the exponents appearing
in a minimal free resolution 
\[ \dots \rightarrow \bigoplus_{q \geq 2} S_\Delta(-q)^{\beta_{2,q}} \rightarrow \bigoplus_{q \geq 1} S_\Delta(-q)^{\beta_{1,q}} \rightarrow \bigoplus_{q \geq 0} S_\Delta(-q)^{\beta_{0,q}} \rightarrow \faktor{S_\Delta}{I_\Delta} \rightarrow 0 \]
of the homogeneous coordinate ring of $X_\Delta$ as a graded $S_\Delta$-module, obtained by taking syzygies.
Traditionally one writes $\beta_{p,p+q}$ in the $p$th column and the $q$th row. 
Alternatively and often more conveniently,
the Betti numbers $\beta_{p,p+q}$ are the dimensions of the Koszul cohomology spaces $K_{p,q}(X_\Delta,  \mathcal{O}(1))$,  
which will be described in detail in Section~\ref{section_koszulcohomology}.

\begin{remark} If $\Delta$ and $\Delta'$ are lattice polygons,
we say that they are \emph{unimodularly equivalent} (denoted by $\Delta \cong \Delta'$)
if they are obtained from one another using a transformation
from the affine group $\AGL_2(\ZZ)$, that is a map of the form
\[
    \RR^2 \rightarrow \RR^2 : (x,y) \mapsto (x,y) A + (a,b) \text{ with } A \in \GL_2(\ZZ) \text{ and } a,b \in \ZZ.
\]
Unimodularly equivalent polygons yield projectively equivalent toric surfaces,
which have the same graded Betti table.
So we are interested in lattice polygons up to unimodular equivalence only.
\end{remark}

In Section~\ref{section_firstfacts} we prove/gather some first facts on the graded Betti table. To begin with, we show 
that it has the following shape:
\begin{lemma} \label{lemmashapebetti} The graded Betti table of $X_\Delta$ has the form
\begin{equation} \label{toricbetti}
\small
\begin{array}{r|ccccccc}
  & 0 & 1 & 2 & 3 & \dots & N_\Delta-4 & N_\Delta-3 \\
\hline
0 & 1 & 0 & 0 & 0 & \dots & 0 & 0 \\
1 & 0 & b_1 & b_2 & b_3 & \dots & b_{N_\Delta-4} & b_{N_\Delta-3} \\
2 & 0 & c_{N_\Delta-3} & c_{N_\Delta-4} & c_{N_\Delta-5} & \dots & c_2 & c_1 \\
\end{array},
\end{equation}
where omitted entries are understood to be $0$. Moreover $(\forall \ell : c_\ell = 0) \Leftrightarrow \Delta^{(1)} = \emptyset$.
\end{lemma}
\noindent We also 
provide a closed formula for the antidiagonal differences:
\begin{lemma} \label{diagonalterms}
For $\ell = 1, \dots, N_\Delta - 2$ one has
\[ b_\ell - c_{N_\Delta - 1 - \ell} = \ell {N_\Delta - 1 \choose \ell + 1} - 2 {N_\Delta - 3 \choose \ell - 1} \vol(\Delta) \]
where it is understood that $b_{N_\Delta - 2} = c_{N_\Delta - 2} = 0$.
\end{lemma}
\noindent This reduces the determination of the graded Betti numbers to that of the $b_i$'s (or of the $c_i$'s).
Finally we give explicit formulas for
the entries $b_1$, $b_2$, and $b_{N_\Delta - 4}$, $b_{N_\Delta - 3}$,
which then also yield explicit descriptions of
$c_1$, $c_2$, $c_3$, and $c_{N_\Delta - 3}$. The precise statements are a bit lengthy and can be found in Section~\ref{subsection_explicitformulas}.

We mentioned Koelman's result on the generators of $I_\Delta$: this was vastly generalized in the Ph.D.\ thesis
of Hering \cite[Thm.\,IV.20]{heringphd}, building on an observation due to Schenck \cite{schenck}
and invoking a theorem of Gallego--Purnaprajna \cite[Thm.\,1.3]{gallego}. She provided a combinatorial interpretation
for the number of leading zeroes in the quadratic strand (the row $q=2$).
\begin{theorem}[Hering, Schenck] \label{heringschenckthm}
If $\Delta^{(1)} \neq \emptyset$ then $\min \{ \, \ell \, | \, c_{N_\Delta-\ell} \neq 0 \, \} = | \partial \Delta \cap \ZZ^2 |$,
where $\partial \Delta$ denotes the boundary of $\Delta$.
\end{theorem}
\noindent In Green's language of property $N_p$, this reads that 
$S_\Delta / I_\Delta$ satisfies $N_p$ if and only if $| \partial \Delta \cap \ZZ^2 | \geq p + 3$. Hering's thesis contains several other statements 
of property $N_p$ type for toric varieties of any dimension.

In Section~\ref{section_Kp1} we work towards a similar combinatorial expression for the number
of zeroes at the end of the linear strand (the row $q=1$). We are unable
to provide a definitive answer, but we formulate a concrete conjecture that we can prove in many special cases.
The central combinatorial notion is the following:
\begin{definition}
Let $\Delta$ be a lattice polygon.
If $\Delta \neq \emptyset$, then the \emph{lattice width} of $\Delta$, denoted $\lw(\Delta)$,
is the minimal height $d$ of a horizontal strip $\RR \times [0,d]$ in which $\Delta$ can
be mapped using a unimodular transformation.
If $\Delta = \emptyset$, we define $\lw(\Delta) = -1$.
\end{definition}
Remark that $\lw(\Delta) = 0$ if and only if $\Delta$ is zero- or one-dimensional.
The lattice width can be computed recursively; see \cite[Thm.\,4]{CaCo} or \cite[Thm.\,13]{lubbesschicho}: if $\Delta$ is two-dimensional then
\[ \lw(\Delta) = \left\{ \begin{array}{ll} \lw(\Delta^{(1)}) + 3 & \text{if $\Delta \cong d \Sigma$ for some $d \geq 2$,} \\ \lw(\Delta^{(1)}) + 2 & \text{if not}, \end{array} \right. \]
where $\Sigma := \conv \{ (0,0), (1,0), (0,1) \}$.

The multiples of $\Sigma$, whose associated toric surfaces are the Veronese surfaces (more precisely $X_{d\Sigma}$ is
the image of $\PP^2$ under the $d$-uple embedding $\nu_d$), will keep playing a special role throughout the rest of this paper. Another important role is attributed to 
multiples of $\Upsilon = \conv \{ (-1,-1), (1,0), (0,1) \}$.
Finally we also introduce the polygons $\Upsilon_d = \conv  \{ (-1,-1), (d,0), (0,d) \}$, where we note that $\Upsilon_1 = \Upsilon$. For the sake of overview,
these polygons are depicted in Figure~\ref{overviewpols}, along with some elementary combinatorial properties.
\begin{figure}[thb]
 \begin{tikzpicture}
   \draw [thick] (0,0) -- (2,0) -- (0,2) -- (0,0);
   \node at (-0.2,-0.25) {\small $(0,0)$};
   \node at (2.2,-0.25) {\small $(d,0)$};
   \node at (-0.2,2.2) {\small $(0,d)$};
   \node at (0.6,0.6) {\small $d \Sigma$};
   \node at (1,-1) {\small $\lw(d\Sigma) = d$};
   \node at (1,-1.5) {\small $(d\Sigma)^{(1)} \cong (d-3)\Sigma$};  
 \end{tikzpicture}
 \qquad \quad 
 \begin{tikzpicture}
   \draw [thick] (-1,-1) -- (1,0) -- (0,1) -- (-1,-1);
   \node at (-1,-1.2) {\small $(-d,-d)$};
   \node at (1.5,0) {\small $(d,0)$};
   \node at (0,1.2) {\small $(0,d)$};
   \node at (0,0) {\small $d \Upsilon$};
   \node at (0,-2) {\small $\lw(d\Upsilon) = 2d$};
   \node at (0,-2.5) {\small $(d\Upsilon)^{(1)} \cong (d-1)\Upsilon$};  
 \end{tikzpicture}
 \qquad \quad 
 \begin{tikzpicture}
   \draw [thick] (-0.333,-0.333) -- (1.666,0) -- (0,1.666) -- (-0.333,-0.333);
   \node at (-0.333,-0.5333) {\small $(-1,-1)$};
   \node at (0.4,0.4) {\small $\Upsilon_d$};
   \node at (0.7,-1.333) {\small $\lw(\Upsilon_d) = d + 1$};
   \node at (0.7,-1.833) {\small $\Upsilon_d^{(1)} \cong (d-1)\Sigma$};
   \node at (0,1.866) {\small $(0,d)$};
   \node at (2.166,0) {\small $(d,0)$};
 \end{tikzpicture}
\caption{Three recurring families of polygons} \label{overviewpols}
\end{figure}
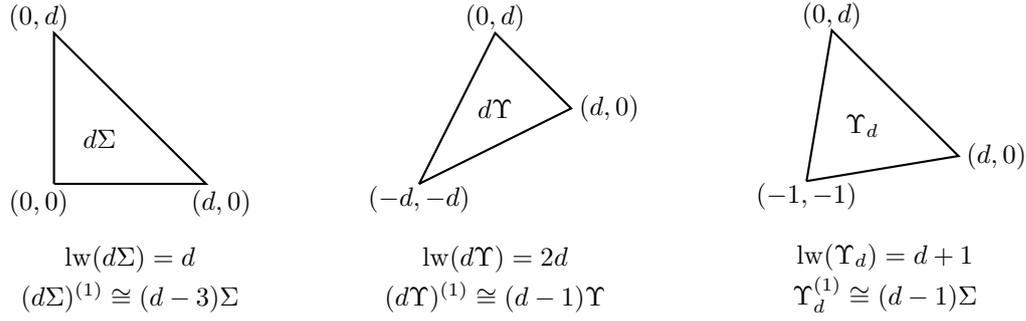

Our conjecture is as follows:
\begin{conjecture} \label{Kp1conjecture}
If $\Delta \not \cong \Sigma, \Upsilon$ then one has $\min \{ \, \ell \, | \, b_{N_\Delta-\ell} \neq 0 \, \} = \lw(\Delta) + 2$, unless
\[ \Delta \cong d\Sigma \text{ for some $d \geq 2$} \ \quad \ \text{or} \ \quad \ \Delta \cong \Upsilon_d \text{ for some $d \geq 2$} \ \quad \ \text{or} \ \quad \ \Delta \cong 2 \Upsilon \]
in which case it is $\lw(\Delta) + 1$.
\end{conjecture}
\noindent In other words we conjecture that the number of zeroes at the end of the linear strand equals $\lw(\Delta) - 1$, unless
$\Delta$ is of the form $d \Sigma$, $\Upsilon_d$ or $2\Upsilon$, in which case it equals $\lw(\Delta) - 2$.

\begin{remark} The excluded cases $\Delta \cong \Sigma, \Upsilon$ are pathological: 
the Betti tables are
\[  \small \begin{array}{r|cc}
  & 0 \\
\hline
0 & 1 \\
1 & 0 \\
2 & 0 \\  
\end{array}
\qquad \text{resp.} \qquad
\begin{array}{r|cc}
  & 0 & 1 \\
\hline
0 & 1 & 0 \\
1 & 0 & 0 \\
2 & 0 & 1 \\  
\end{array}, \]
i.e.\ the entire linear strands are zero. \end{remark}

As explained in Section~\ref{section_Kp1} the upper bound $\min \{ \, \ell \, | \, b_{N_\Delta-\ell} \neq 0 \, \} \leq \lw(\Delta) + 2$
follows from the fact that our toric surface $X_\Delta$ is naturally contained
in a rational normal scroll of dimension $\lw(\Delta) + 1$, which is known to have 
non-zero linear syzygies up to
column $p = N_\Delta - \lw(\Delta) - 2$. Then also $X_\Delta$ must have non-zero linear syzygies up to that point,
yielding the desired bound. Thus another way of reading Conjecture~\ref{Kp1conjecture}
is that the natural bound coming from this ambient rational normal scroll is usually sharp.
This is in the philosophy of Green's $K_{p,1}$ theorem \cite[Thm.\,3.31]{nagelaprodu} that
towards the end of the resolution, `most' linear syzygies must come from the smallest ambient variety of minimal degree. 
 In the exceptional cases $d \Sigma$, $\Upsilon_d$ and $2\Upsilon$ we
can prove the sharper bound $\min \{ \, \ell \, | \, b_{N_\Delta-\ell} \neq 0 \, \} \leq \lw(\Delta) + 1$ by following a slightly different 
argument, using explicit computations in Koszul cohomology.


We can prove sharpness of these bounds in a considerable number of special situations, overall
leading to the following partial
result towards Conjecture~\ref{Kp1conjecture}.

\begin{theorem} \label{theoremKp1progress}
If $\Delta \not \cong \Sigma, \Upsilon$ then one has $\min \{ \, \ell \, | \, b_{N_\Delta-\ell} \neq 0 \, \} \leq \lw(\Delta) + 2$. If
\[ \Delta \cong d\Sigma \text{ for some $d \geq 2$} \ \quad \ \text{or} \ \quad \ \Delta \cong \Upsilon_d \text{ for some $d \geq 2$} \ \quad \ \text{or} \ \quad \ \Delta \cong 2 \Upsilon \]
then moreover one has the sharper bound $\lw(\Delta) + 1$. In other words the sharpest applicable upper bound predicted
by Conjecture~\ref{Kp1conjecture} holds. Moreover:
\begin{itemize}
  \item If $N_\Delta \leq 32$ then the bound is met.
  \item If a certain non-exceptional lattice polygon $\Delta$ (i.e.\ not of the form $d \Sigma, \Upsilon_d, 2\Upsilon$) meets the bound then so does 
  every lattice polygon containing
  $\Delta$ and having the same lattice width.
    In particular if $\lw(\Delta) \leq 6$ then the bound is met.
  
  \item[$\dagger$] If $\Delta = \Gamma^{(1)}$ for some larger lattice polygon $\Gamma$ and if Green's canonical syzygy conjecture
  holds for smooth curves on $X_\Delta$ (known to be true if $H^0(X_\Delta, -K_{X_\Delta}) \geq 2$) then the bound is met.
  
\end{itemize}
\end{theorem}

\noindent Sharpness in the cases where $N_\Delta \leq 32$ is obtained by explicit verification, based on the data from~\cite{movingout} 
and using the algorithm described in Section~\ref{section_computing}; this covers 
more than half a million (unimodular equivalence classes of) small lattice polygons. 
The statement involving $\lw(\Delta) \leq 6$ relies on this exhaustive verification, along with the classification of
inclusion-minimal lattice polygons having a given lattice width, which is elaborated in~\cite{coolslemmens}.


\begin{remark}
The statement marked with $\dagger$ will not be proven in the current paper, 
even though it is actually the reason why we came up with Conjecture~\ref{Kp1conjecture} in the first place. To date,
Green's canonical syzygy conjecture for curves in toric surfaces remains open in general, 
but the cases where $H^0(X_\Delta, -K_{X_\Delta}) \geq 2$  
are covered by recent work of Lelli-Chiesa~\cite{lellichiesa},
which allows one to deduce Conjecture~\ref{Kp1conjecture} for all multiples of $\Upsilon$, for all multiples of $\Sigma$, 
for all polygons $[0,a] \times [0,b]$ with $a,b \geq 1$, and so on. 
The details of this are discussed in a subsequent paper~\cite{nextpaper}, which is devoted to syzygies of curves in toric surfaces.
For the sake of conciseness
we have chosen to keep the present document curve-free.
\end{remark}

Next we describe our algorithm for determining 
the graded Betti table of $X_\Delta \subseteq \PP^{N_\Delta - 1}$ upon input of a lattice polygon $\Delta$, by explicitly computing its Koszul cohomology. The
details can be found in Section~\ref{section_computing}, but in a nutshell the ingredients are as follows.
The most dramatic speed-up comes from incorporating the torus action, which decomposes the cohomology spaces into eigenspaces, one for each bidegree $(a,b) \in \mathbb{Z}^2$, all but finitely many of which are trivial. 
Another important speed-up comes from toric Serre duality, enabling a meet-in-the-middle approach where one fills the graded Betti table starting from
the left and from the right simultaneously. A third speed-up comes from the explicit formula for the antidiagonal differences given in Lemma~\ref{diagonalterms}, thanks
to which it suffices to determine half of the graded Betti table only. Moreover if $\left| \partial \Delta \cap \ZZ^2 \right|$ is large (which is particularly the case for
the Veronese polygons $d\Sigma$) then many
of these entries come for free using Hering and Schenck's Theorem~\ref{heringschenckthm}.
A fourth theoretical ingredient is a combinatorial description of certain exact subcomplexes of the Koszul complex
that can be quotiented out, resulting in smaller vector spaces, thereby making the linear algebra more manageable. Because this seems interesting in its own right, we
have devoted the separate Section~\ref{section_alexandertruck} to it. Final ingredients include sparse linear algebra, using symmetries, and working in finite characteristic. More precisely, 
most of the data gathered in this article, some of which can be found in Appendix~\ref{appendix_data}, are obtained by computing modulo $40\,009$, the smallest prime
number larger than $40\,000$.

\begin{remark} \label{remark_semicontinuity}
By semi-continuity the entries of the graded Betti table cannot decrease upon reduction of $X_\Delta$ modulo some prime number. Therefore
working in finite characteristic is fine for proving that certain entries are zero, as is done in our partial verification of Conjecture~\ref{Kp1conjecture}. But entries that are found to be non-zero
might a priori be too large, even though we do not expect them to be. Therefore the non-zero entries of some of the graded Betti tables given in Appendix~\ref{appendix_data} are conjectural.
For technical reasons our current implementation does not straightforwardly adapt to characteristic zero, but we are working on fixing this issue.
Although it would come at the cost of some efficiency, this should enable us to confirm all of the data from Appendix~\ref{appendix_data} in
characteristic zero.
\end{remark}

In view of the wide interest in syzygies of Veronese modules~\cite{brunsconca,EinErmanLazarsfeld,greco,loose,ottaviani,park,rubei}, the 
most interesting new graded Betti table that we obtain is that of $X_{6\Sigma} \subseteq \PP^{27}$, i.e.\ the image of $\PP^2$ under the $6$-uple embedding $\nu_6$, in characteristic $40\,009$.
Up to $5\Sigma$ this data was recently gathered (in characteristic zero) by Greco and Martino~\cite{greco}. An extrapolating glance at these Betti tables naturally leads to the following conjecture:

\begin{conjecture} \label{veronese_conjecture}
Consider the graded Betti table of the $d$-fold Veronese surface $X_{d \Sigma}$. If $d \geq 2$ then the last
non-zero entry on the linear strand is
\[ b_{d(d+1)/2} = \frac{d^3(d^2 - 1)}{8}, \]
while if $d \geq 3$ then the first non-zero entry on the quadratic strand is
\[ c_g = { N_{(d\Sigma)\interior} + 8 \choose 9}  \]
where $N_{(d\Sigma)\interior} = |�(d \Sigma)^{(1)} \cap \ZZ^2 | = (d-1)(d-2)/2$.
\end{conjecture}

\subsubsection*{Acknowledgements}

This research was partially supported by the research project G093913N of the Research Foundation Flanders (FWO),
by the European Research Council under the European Community's Seventh Framework Programme (FP7/2007-2013) with ERC Grant Agreement nr. 615722
MOTMELSUM, and by the Labex CEMPI  (ANR-11-LABX-0007-01).
The fourth author is supported by a Ph.D.\ fellowship of the Research Foundation Flanders (FWO).
We would like to thank Milena Hering and Nicolas M.~Thi\'ery for several helpful remarks.
The computational resources (Stevin Supercomputer Infrastructure)
and services used in this work were provided by the VSC (Flemish Supercomputer Center),
funded by Ghent University, the Hercules Foundation and the Flemish Government --- department EWI.
After submitting a first version of this paper to \texttt{arXiv}, we learned that a group of researchers at
the University of Wisconsin-Madison has been working independently on computing Betti tables
of Veronese surfaces~\cite{wisconsinref}, thereby obtaining results that partially overlap with our own observations.
In particular they also obtain the graded Betti table of $X_{6\Sigma} \subseteq \PP^{27}$, although
here too the result is conjectural, using linear algebra over the reals rather than mod $p$.
We thank David J.\ Bruce for getting in touch with us on this, and for his valuable feedback.

\section{Koszul cohomology of toric surfaces}\label{section_koszulcohomology}

As is well-known, instead of using syzygies, the entries of the graded Betti table can also be defined
as dimensions of Koszul cohomology spaces, which we now explicitly describe in the specific case of toric surfaces.
We refer to the book by Aprodu and Nagel~\cite{nagelaprodu} for an introduction to Koszul cohomology, and to the books by Fulton~\cite{fulton}
and Cox, Little and Schenck~\cite{coxlittleschenck} for more background on toric geometry.

For a lattice polygon $\Delta$ we write $V_\Delta$ for the space of Laurent polynomials 
\[ \sum_{(i,j) \in \Delta \cap \ZZ^2} c_{i,j}x^i y^j \ \in \, k[x^{\pm 1}, y^{\pm 1}], \]
which we view as functions on $X_\Delta$ through $\varphi_\Delta$.
This equals the space $H^0(X_\Delta,L_\Delta)$ of global sections of $\mathcal{O}(L_\Delta)$, where $L_\Delta$ is 
some concrete very ample torus-invariant divisor on $X_\Delta$ satisfying $\mathcal{O}(L_\Delta) \cong \mathcal{O}(1)$.
More generally $V_{q\Delta} = H^0(X_\Delta,qL_\Delta)$ for each $q \geq 0$.

Then the entry in the $p$th column and the $q$th row of the graded Betti table of $X_\Delta$
is the dimension of the Koszul cohomology space $K_{p,q}(X_\Delta,L_\Delta)$, defined as the cohomology in the middle of
\begin{multline*}
    \wedgepow{p+1} H^0(X_\Delta,L_\Delta) \otimes H^0(X_\Delta,(q-1)L_\Delta)
    \stackrel{\delta}{\longrightarrow} 
    \wedgepow{p} H^0(X_\Delta,L_\Delta) \otimes H^0(X_\Delta,qL_\Delta) \\
    \stackrel{\delta'}{\longrightarrow}
    \wedgepow{p-1} H^0(X_\Delta,L_\Delta) \otimes H^0(X_\Delta,(q+1)L_\Delta)
\end{multline*}
which can be rewritten as
\begin{equation}\label{cohompq}
\wedgepow{p+1} V_\Delta \otimes V_{(q-1)\Delta}
\stackrel{\delta}{\longrightarrow}
\wedgepow{p} V_\Delta \otimes V_{q\Delta}
\stackrel{\delta'}{\longrightarrow}
\wedgepow{p-1} V_\Delta \otimes V_{(q+1)\Delta}.
\end{equation}
Here the coboundary maps
$\delta$ and $\delta'$ are defined by
\begin{equation}\label{delta}
    v_1 \wedge v_2 \wedge v_3 \wedge v_4 \wedge \dots \otimes w \, \mapsto \, \sum (-1)^s v_1 \wedge v_2 \wedge v_3 \wedge v_4 \wedge \dots \wedge \widehat{v_s} \wedge \dots \otimes v_s w
\end{equation}
where $s$ ranges from $1$ to $p+1$ resp.\ $1$ to $p$, and $\widehat{v_s}$ means that $v_s$ is being omitted. 
In particular one sees that $b_\ell$ is the dimension of the cohomology in the middle of
\begin{equation}\label{cohombell}
    \wedgepow{\ell+1} V_\Delta
    \stackrel{\delta}{\longrightarrow}
    \wedgepow{\ell} V_\Delta \otimes V_{\Delta} 
    \stackrel{\delta'}{\longrightarrow}
    \wedgepow{\ell-1} V_\Delta \otimes V_{2\Delta},
\end{equation}
where we note that the left map is always injective. On the other hand
$c_\ell$ is the dimension of the cohomology in the middle of
\begin{equation}\label{cohomcell}
    \wedgepow{N_\Delta - 1 - \ell} V_\Delta \otimes V_{\Delta}
    \stackrel{\delta}{\longrightarrow}
    \wedgepow{N_\Delta - 2 - \ell} V_\Delta \otimes V_{2\Delta} 
    \stackrel{\delta'}{\longrightarrow}
    \wedgepow{N_\Delta - 3 - \ell} V_\Delta \otimes V_{3\Delta},
\end{equation}
for all $\ell = 1, \dots, N_\Delta - 3$. 

\subsection{Duality} A more concise description of the $c_\ell$'s is obtained using Serre duality.
Because the version that we will invoke
requires us to work with 
smooth surfaces, we consider a toric resolution of singularities
$X \rightarrow X_\Delta$ and let $L$ be the pullback of $L_\Delta$. Then $L$ may no longer be very ample,
but it remains globally generated by the same global sections $V_\Delta$. Let $K$ be the canonical divisor on $X$ obtained
by taking minus the sum of all torus-invariant prime divisors. 
By Demazure vanishing one has $H^1(X,qL) = 0$ for all $q \geq 0$, 
so that we can apply
the duality formula from~\cite[Thm.\,2.25]{nagelaprodu}, which in our case reads
\[ K_{p,q}(X,L)^\vee \cong K_{N_\Delta-3-p,3-q}(X;K,L), \]
to conclude that
\begin{align*}
    b_\ell &= \dim K_{\ell,1}(X_\Delta, L_\Delta) = \dim K_{\ell,1}(X, L) = \dim K_{N_\Delta - 3 - \ell,2}(X;K,L), \\
    c_\ell &= \dim K_{N_\Delta - 2 - \ell,2}(X_\Delta, L_\Delta) = \dim K_{N_\Delta - 2 - \ell,2}(X, L) = \dim K_{\ell - 1,1}(X;K,L),
\end{align*}
again for all $\ell = 1, \dots, N_\Delta - 3$. 
Here the attribute `$;K$' denotes Koszul cohomology twisted by $K$, which is defined as before, except that
each appearance of $\cdot \otimes H^0(X,qL)$ is replaced by $\cdot \otimes H^0(X,qL + K)$.
Using that $H^0(X,qL + K) = V_{(q\Delta)^{(1)}}$ for $q \geq 1$ and that
$H^0(X,K) = 0$ we find that $b_\ell$ is the cohomology in the middle of
\begin{equation} \label{cohombelldual}
    \wedgepow{N_\Delta - 2 - \ell} V_\Delta \otimes V_{\Delta^{(1)}}
    \stackrel{\delta}{\longrightarrow}
    \wedgepow{N_\Delta - 3 - \ell} V_\Delta \otimes V_{(2\Delta)^{(1)}} 
    \stackrel{\delta'}{\longrightarrow}
    \wedgepow{N_\Delta - 4 - \ell} V_\Delta \otimes V_{(3\Delta)^{(1)}}
\end{equation}
and, more interestingly, that $c_\ell$ is the dimension of the kernel of
\begin{equation} \label{cohomcelldual}
    \wedgepow{\ell-1} V_\Delta \otimes V_{\Delta^{(1)}} 
    \stackrel{\delta'}{\longrightarrow}
    \wedgepow{\ell-2} V_\Delta \otimes V_{(2\Delta)^{(1)}}.
\end{equation}
For example this gives a quick way of seeing that $c_1 = \dim \ker (V_{\Delta^{(1)}} \rightarrow 0) = N_{\Delta^{(1)}}$.

\subsection{Bigrading}\label{section_bigrading}

For $(a,b) \in \ZZ^2$ we call an element of 
\[ \wedgepow{p} V_\Delta \otimes V_{q\Delta} \]
homogeneous of bidegree $(a,b)$ if it is a $k$-linear combination of elementary tensors of the form
\[ x^{i_1}y^{j_1} \wedge \dots \wedge x^{i_p}y^{j_p} \otimes x^{i'}y^{j'} \]
satisfying $(i_1,j_1) + \dots + (i_p, j_p) + (i',j') = (a,b)$. The coboundary morphisms $\delta$ and $\delta'$
send homogeneous elements to homogeneous elements of the same bidegree, i.e.\ the Koszul complex is naturally bigraded. Thus
the Koszul cohomology spaces decompose as
\[ K_{p,q}(X,L) = \bigoplus_{(a,b) \in \ZZ^2}  K_{p,q}^{(a,b)}(X,L) \]
where in fact it suffices to let $(a,b)$ range over $(p+q)\Delta \cap \ZZ^2$.
Similarly, we have
a decomposition of the twisted cohomology spaces
\[ K_{p,q}(X; K,L) = \bigoplus_{(a,b) \in \ZZ^2}  K_{p,q}^{(a,b)}(X; K,L) \]
where now $(a,b)$ in fact runs over $\left( p\Delta + (q\Delta)^{(1)} \right) \cap \ZZ^2$. In particular also the
$b_\ell$'s and the $c_\ell$'s, and as a matter of fact the entire graded Betti table, decompose as sums of smaller instances.
We will write
\begin{align*}
    b_{\ell,(a,b)} &= \dim K_{\ell,1}^{(a,b)}(X,L), & b_{\ell,(a,b)}^\vee &= \dim K_{N_\Delta-3-\ell,2}^{(a,b)}(X;K,L), \\
    c_{\ell,(a,b)} &= \dim K_{N_\Delta-2-\ell,2}^{(a,b)}(X,L), & c_{\ell,(a,b)}^\vee &= \dim K_{\ell-1,1}^{(a,b)}(X;K,L),
\end{align*}
so that
\[ b_\ell = \sum_{(a,b) \in \ZZ^2} b_{\ell, (a,b)} = \sum_{(a,b) \in \ZZ^2} b_{\ell, (a,b)}^\vee \quad \text{and} \quad
c_\ell = \sum_{(a,b) \in \ZZ^2} c_{\ell, (a,b)} = \sum_{(a,b) \in \ZZ^2} c_{\ell, (a,b)}^\vee. \]

\begin{example}
For $\Delta = 4\Sigma$ one can compute that $c_3 = \dim K_{2,1}(X;K,L) = 55$, 
which decomposes as the sum of the following numbers.
\[
\footnotesize
\begin{array}{llllllllll}
0 \\
0 & 0 \\
0 & 1 & 0 \\
0 & 1 & 1 & 0 \\
0 & 2 & 2 & 2 & 0 \\
0 & 2 & 3 & 3 & 2 & 0 \\
0 & 2 & 3 & 4 & 3 & 2 & 0 \\
0 & 1 & 2 & 3 & 3 & 2 & 1 & 0 \\
0 & 1 & 1 & 2 & 2 & 2 & 1 & 1 & 0 \\
0 & 0 & 0 & 0 & 0 & 0 & 0 & 0 & 0 & 0 \\
\end{array}
\]
Here the entry in the $a$th column (counting from the left) and the $b$th row (counting from the bottom) is the dimension 
$c^\vee_{3,(a,b)}$ of the degree $(a,b)$ part.
In other words we think of the above triangle as being in natural correspondence with the lattice points $(a,b)$ inside 
$2\Delta + \Delta^{(1)} = (1,1) + 9\Sigma$.   
\end{example}

\subsection{Duality versus bigrading}

An interesting observation that came out of a joint discussion with
Milena Hering is that duality respects the bigrading along the rule
\[ K_{p,q}^{(a,b)}(X,L)^\vee \cong K_{N_\Delta-3-p,3-q}^{\sigma_\Delta - (a,b)}(X;K,L), \]
where $\sigma_\Delta$ denotes the sum of all lattice points in $\Delta$.
We postpone a proof to~\cite{boocherhering}, but note that
taking dimensions yields the formulas
\begin{equation} \label{dualgradingformulas}
 b_{\ell,(a,b)} = b_{\ell, \sigma_\Delta - (a,b)}^\vee \quad \text{and} \quad c_{\ell,(a,b)} = c_{\ell, \sigma_\Delta - (a,b)}^\vee.
\end{equation}
These imply that $K_{p,q}(X,L)$ is actually supported on the degrees $(a,b)$ that are contained in
\[   (p+q)\Delta \ \cap \ \left( \sigma_\Delta - (N_\Delta - 3 - p)\Delta  - ( (3-q)\Delta )^{(1)} \right),  \]
and similarly that $K_{p,q}(X;K,L)$ vanishes outside
\[ \left( p\Delta + (q\Delta)^{(1)} \right) \cap \left( \sigma_\Delta - (N_\Delta - p - q) \Delta  \right). \]
The image below illustrates this for $\Delta = 2\Upsilon$, $p=4$, $q=1$, where $K_{p,q}(X;K,L)$ is supported on $9\Upsilon \cap (-10\Upsilon)$:
\begin{center}
 \begin{tikzpicture}[scale=0.2]
   \filldraw[fill=black, fill opacity=0.1, thick] (9,9) -- (-9,0) -- (0,-9) -- (9,9);
   \filldraw[fill=black, fill opacity=0.1, thick] (-10,-10) -- (10,0) -- (0,10) -- (-10,-10);      
   \node[coordinate, label=left:{\small $(-10,10)$}] at (-10,-10) {};
   \node[coordinate, label=right:{\small $(10,0)$}] at (10,0) {};
   \node[coordinate, label=above:{\small $(0,10)$}] at (0,10) {};
   \node[coordinate, label=right:{\small $(9,9)$}] at (9,9) {};
   \node[coordinate, label=left:{\small $(-9,0)$}] at (-9,0) {};
   \node[coordinate, label=below:{\small $(0,-9)$}] at (0,-9) {};
 \end{tikzpicture}
\end{center}
In principle
this could be used to speed up our computation of the graded Betti table, 
because it says that certain bidegrees can be omitted.
Unfortunately the vanishing happens in a range of bidegrees that is dealt with relatively easily anyway.
Therefore, the computational advantage is negligible and
we will not use this in our algorithm. 

\section{First facts on the graded Betti table} \label{section_firstfacts}

\subsection{Overall shape of the graded Betti table}

We prove the shape of the graded Betti table of $X_\Delta$ announced in Lemma~\ref{lemmashapebetti}, by invoking
some well-known theorems from the existing literature.
It is also possible to give a more elementary, handcrafted proof using Koszul cohomology.
\begin{proof}[Proof of Lemma~\ref{lemmashapebetti}]
  Hochster has proven that $S_\Delta / I_\Delta$ is a Cohen--Macaulay module 
  \cite[Ex.\,9.2.8]{coxlittleschenck}. Its Krull dimension equals $3$, and therefore the Auslander--Buchsbaum
  formula \cite[Thm.\,A.2.15]{eisenbud} implies that the graded Betti table has non-zero entries up to column
  $p = N_\Delta - 3$.  
  Now it is well-known 
  that the Hilbert polynomial $P_{X_\Delta}(d)$ of $X_\Delta$ is given by the Ehrhart polynomial
  \begin{equation} \label{hilbertpol}
    | d\Delta \cap \ZZ^2| = \vol(\Delta) d^2 + \frac{|\partial \Delta \cap \ZZ^2|}{2} d + 1,
  \end{equation}
  and that this matches with the Hilbert function $H_{X_\Delta}(d)$ for all integers $d \geq 0$. In fact, the smallest
  integer $s$ such that $P_{X_\Delta}(d) = H_{X_\Delta}(d)$ for all $d \geq s$ is
  \[ \left\{ \begin{array}{cl} 0 & \text{if $\Delta^{(1)} \neq \emptyset$,} \\ 
   -1 & \text{if $\Delta^{(1)} = \emptyset$.} \\ 
   \end{array} \right. \] 
  From \cite[Cor.\,4.8]{eisenbud} we conclude that the Castelnuovo--Mumford 
  regularity of $X_\Delta$ equals $2$, unless $\Delta^{(1)} = \emptyset$ in which case it equals $1$.
\end{proof}

The polygons for which $\Delta^{(1)} = \emptyset$ have the following geometric characterization:
\begin{lemma} \label{minimaldegreelemma}
The surface $X_\Delta \subseteq \PP^{N_\Delta - 1}$ is a variety of minimal degree
if and only if $\Delta^{(1)} = \emptyset$.
\end{lemma}
\begin{proof}
By definition $X_\Delta$ has minimal degree if and only if $\deg X_\Delta = 1 + \codim X_\Delta$.
By the above formula \eqref{hilbertpol} for the Hilbert polynomial this can be rewritten as
\[ 2 \vol(\Delta) = N_\Delta - 2 \]
which by Pick's theorem holds if and only if $\Delta^{(1)} = \emptyset$.
\end{proof}

It follows that if $\Delta^{(1)} = \emptyset$ 
then the graded Betti table of $X_\Delta$
is of the form
\begin{equation} \label{toricbettileeginwendige}
\small
\begin{array}{r|ccccccc}
  & 0 & 1 & 2 & 3 & \dots & N_\Delta-4 & N_\Delta-3 \\
\hline
0 & 1 & 0 & 0 & 0 & \dots & 0 & 0 \\
1 & 0 & {N_\Delta - 2 \choose 2} & 2 {N_\Delta - 2 \choose 3} & 3 {N_\Delta - 2 \choose 4} & \dots & (N_\Delta-4) {N_\Delta - 2 \choose N_\Delta-3} & (N_\Delta - 3) {N_\Delta - 2 \choose N_\Delta - 2}, \\
\end{array}
\end{equation}
because the Eagon--Northcott complex is exact in this case; see for instance \cite[App.\,A2H]{eisenbud}. It also follows that 
if $\Delta^{(1)} \neq \emptyset$ then
$b_{N_\Delta - 3} = 0$; see \cite[Thm.\,3.31(i)]{nagelaprodu}.
From a combinatorial viewpoint the two-dimensional lattice polygons $\Delta$ for
which $\Delta^{(1)} = \emptyset$ were classified in \cite[Ch.\,4]{koelmanthesis}: up to unimodular equivalence 
they are $2\Sigma$ and
the Lawrence prisms 
\begin{center}
 \begin{tikzpicture}[scale=0.8]
   \draw [thick] (0,0) -- (4,0) -- (3,0.6) -- (0,0.6) -- (0,0);
    \node at (-0.2,-0.3) {\small $(0,0)$};
    \node at (4.1,-0.3) {\small $(a,0)$};
    \node at (3.1,0.9) {\small $(b,1)$};
    \node at (-0.2,0.9) {\small $(0,1)$};
    \node at (9,0.3) {for integers $a \geq b \geq 0$ with $a > 0$.};
 \end{tikzpicture}
 \end{center}
The respective corresponding $X_\Delta$'s are the Veronese surface in $\PP^5$ and the rational normal surface
scrolls of type $(a,b)$.
One thus sees that Conjecture~\ref{Kp1conjecture} is true if $\Delta^{(1)} = \emptyset$.

\subsection{Antidiagonal differences}\label{section_diagonal_differences}

From the explicit shape \eqref{hilbertpol} of the Hilbert polynomial,
the closed formula 
\[ b_\ell - c_{N_\Delta - 1 - \ell} = \ell {N_\Delta - 1 \choose \ell + 1} - 2 {N_\Delta - 3 \choose \ell - 1} \vol(\Delta) \]
for the antidiagonal differences, which was announced in Lemma~\ref{diagonalterms}, can be proved by induction.
We will give a slightly more convenient argument using Koszul cohomology.
\begin{proof}[Proof of Lemma~\ref{diagonalterms}.]
The proof relies on three elementary facts: 
\begin{enumerate}
\item[(i)] Pick's theorem, 
\item[(ii)] for any bounded complex of finite-dimensional vector spaces $V_j$ one has 
\[ \sum_j(-1)^j \dim V_j=\sum_j(-1)^j \dim H^j,\] 
where $H^j$ is the cohomology of the complex at place $j$,
\item[(iii)] for all $n,k,N\geq 0$ we have $\sum_{j=0}^n(-1)^j\binom{N}{n-j}\binom{j}{k}=(-1)^k\binom{N-k-1}{n-k}$.
\end{enumerate}
We compute
\begin{align*}
b_\ell-c_{N_\Delta-1-\ell}&=\sum_{j=0}^{\ell+1}(-1)^{j+1}\dim K_{\ell-j+1,j}(X_\Delta,L_\Delta) \\
&\overset{(\text{ii})}{=}\sum_{j=0}^{\ell+1}(-1)^{j+1}\dim\left(\wedgepow{\ell+1-j} V_\Delta\otimes V_{j\Delta}\right) \\
&=\sum_{j=0}^{\ell+1}(-1)^{j+1}\binom{N_\Delta}{\ell+1-j}N_{j\Delta} \\
&\overset{(\text{i})}{=}-\sum_{j=0}^{\ell+1}(-1)^j\binom{N_\Delta}{\ell+1-j}(j^2\vol(\Delta)+\frac{j}{2}\left|\partial\Delta \cap \mathbb{Z}^2\right| +1) \\
&\overset{(\text{i})}{=}-\sum_{j=0}^{\ell+1}(-1)^j\binom{N_\Delta}{\ell+1-j}(j^2\vol(\Delta)+j(N_\Delta-\vol(\Delta)-1)+1) \\
&=-\sum_{j=0}^{\ell+1}(-1)^j\binom{N_\Delta}{\ell+1-j}\left(2\vol(\Delta)\binom{j}{2}+(N_\Delta-1)\binom{j}{1}+\binom{j}{0}\right) \\
&\overset{(\text{iii})}{=}-2\vol(\Delta)\binom{N_\Delta-3}{\ell-1}+(N_\Delta-1)\binom{N_\Delta-2}{\ell}-\binom{N_\Delta-1}{\ell+1} \\
&=-2\vol(\Delta)\binom{N_\Delta-3}{\ell-1}+\ell\binom{N_\Delta-1}{\ell+1},
\end{align*} which equals the desired expression.
\end{proof}

We note the following corollary to Lemma~\ref{diagonalterms}:

\begin{corollary} For all $\ell$ one has that $b_\ell \geq c_{N_\Delta-1-\ell}$
if and only if 
\[ \ell \leq \frac{ (N_\Delta-1)(N_\Delta-2) }{2 \vol(\Delta)} - 1.\]
\end{corollary}
\begin{remark} Note that 
$2\vol(\Delta) = 2N_\Delta - |\partial\Delta \cap \ZZ^2| - 2$ by Pick's theorem. This is typically $\approx 2N_\Delta$,
so the point where the $c_\ell$'s take over from the $b_\ell$'s
is about halfway the Betti table. 
If $|\partial\Delta \cap \ZZ^2|$ is relatively large then
$2\vol(\Delta)$ becomes smaller when compared to $N_\Delta$,
and the takeover point is shifted to the right.
\end{remark}

\subsection{Explicit formulas for some entries} \label{subsection_explicitformulas}

We can give a complete combinatorial characterization of eight entries. Six of these
are rather straightforward:

\begin{corollary}
On the quadratic strand one has
\[ c_1 = N_{\Delta^{(1)}}, \qquad \qquad 
c_2 = \left\{ \begin{array}{ll} (N_\Delta - 3)(N_{\Delta^{(1)}} - 1) & \text{if $\Delta^{(1)} \neq \emptyset$,} \\ 
0 & \text{if $\Delta^{(1)} = \emptyset$,} \\  \end{array} \right. \]
\[ c_{N_\Delta - 3} = \left\{ \begin{array}{ll} 0 & \text{if $|\partial \Delta \cap \ZZ^2| > 3$,} \\ 1 & \text{if $|\partial \Delta \cap \ZZ^2| = 3$ and $\dim \Delta^{(1)} = 2$,} \\ N_\Delta - 3 & \text{if $|\partial \Delta \cap \ZZ^2| = 3$ and $\dim \Delta^{(1)} \leq 1$.} \\ \end{array}
\right.  \] 
On the linear strand one has
\[ b_1 = {N_\Delta - 1 \choose 2} - 2 \vol(\Delta), 
\qquad \qquad b_{N_\Delta - 3} = \left\{ \begin{array}{ll} 0 & \text{if $\Delta^{(1)} \neq \emptyset$}, \\ N_\Delta - 3 & \text{if $\Delta^{(1)} = \emptyset$,} \\  \end{array} \right. \]
\[ b_2 = 2 {N_\Delta - 1 \choose 3} - 2(N_\Delta - 3) \vol(\Delta) + c_{N_\Delta - 3}. \]
\end{corollary}
\begin{proof}
The formulas for $b_1$ and $c_1$ follow immediately from Lemma~\ref{diagonalterms}, where in the latter case
we use that $N_\Delta - 2 - 2\vol(\Delta) = N_{\Delta^{(1)}}$ by Pick's theorem. The entry
$c_{N_\Delta - 3}$ equals the number of cubics in a minimal set of generators of $I_\Delta$, which
was determined in \cite[\S2]{canonical}. Together with Lemma~\ref{diagonalterms} this then gives the formula for $b_2$.
The formula for $b_{N_\Delta - 3}$ was discussed above, and the formula for $c_2$ then again follows using Lemma~\ref{diagonalterms}
in combination with Pick's theorem.
\end{proof}

In Section~\ref{section_formulabN4} we will extend this list as follows. 
This will take considerably more work, and depends on our proof of Conjecture~\ref{Kp1conjecture}
for polygons of small lattice width, given in Section~\ref{section_deformation}. 
\begin{theorem} \label{thirdentry}
Assume that $N_\Delta \geq 4$, or equivalently that $\Delta \not \cong \Sigma$. Then we have
\[ b_{N_\Delta - 4} = (N_\Delta - 4) \cdot B_\Delta \quad \text{where} \quad B_\Delta = \left\{  \begin{array}{ll} 0 & \text{if $\dim \Delta^{(1)} = 2$, $\Delta\not \cong\Upsilon_2$}, \\ 1 & \text{if $\dim \Delta^{(1)} = 1$ or $\Delta \cong \Upsilon_2$}, \\ (N_\Delta - 1)/2 & \text{if $\dim \Delta^{(1)} = 0$}, \\
N_\Delta - 2 & \text{if $\Delta^{(1)} = \emptyset$} \\ \end{array} \right. \]
and
\[ c_3 = (N_\Delta - 4) \left( (N_\Delta - 3)  \vol (\Delta) - \frac{(N_\Delta - 1)(N_\Delta - 2)}{2} + B_\Delta \right). \]
\end{theorem}

\section{Bound on the length of the linear strand} \label{section_Kp1}

\subsection{Bound through rational normal scrolls}

Let $\Delta \subseteq \RR^2$ be a two-dimensional lattice polygon and apply a 
unimodular transformation in order to have $\Delta \subseteq \RR \times [0,d]$ with $d = \lw(\Delta)$.
For each $j = 0, \dots, d$ consider 
\[ m_j = \min \{ a \, | \, (a,j) \in \Delta \cap \ZZ^2 \} \quad \text{and} \quad M_j = \max \{ a \, | \, (a,j) \in \Delta \cap \ZZ^2 \}. \]
These are well-defined, i.e.\ on each height $j$ there is at least one lattice point in $\Delta$, see for instance \cite[Lem.\,5.2]{linearpencils}. Recall that $X_\Delta$ is the Zariski closure of the image of
\begin{align*}
\varphi_\Delta : (k^*)^2 \hookrightarrow \mathbb{P}^{ N_\Delta - 1} :  (\alpha, \beta) \mapsto (&\alpha^{m_0} \beta^0, \alpha^{m_0+1} \beta^0, \dots, \alpha^{M_0} \beta^0, \\
  & \alpha^{m_1} \beta^1, \alpha^{m_1 + 1} \beta^1, \dots, \alpha^{M_1} \beta^1, \\
  & \hspace{25mm}\vdots \\
  & \alpha^{m_d} \beta^d, \alpha^{m_d + 1} \beta^d, \dots, \alpha^{M_d} \beta^d).
\end{align*}
It is clear that this is contained in the Zariski closure of the image of
\begin{align*}
(k^*)^{1 + d} \hookrightarrow \mathbb{P}^{ N_\Delta - 1} : (\alpha, \beta_1, \dots, \beta_d) \mapsto (&\alpha^{m_0} \beta_0, \alpha^{m_0+1} \beta_0, \dots, \alpha^{M_0} \beta_0, \\
  & \alpha^{m_1} \beta_1, \alpha^{m_1 + 1} \beta_1, \dots, \alpha^{M_1} \beta_1, \\
  & \hspace{25mm}\vdots \\
  & \alpha^{m_d} \beta_d, \alpha^{m_d + 1} \beta_d, \dots, \alpha^{M_d} \beta_d)
\end{align*}
where $\beta_0 = 1$.
This is a $(d+1)$-dimensional rational normal scroll, spanned by rational normal curves of degrees $M_0 - m_0$, $M_1 - m_1$, \dots, $M_d - m_d$ (some of these degrees
may be zero,
in which case the `curve' is actually a point).
Its ideal is obtained from $I_\Delta$ by restricting to those binomial generators that
remain valid if one forgets about the vertical structure of $\Delta$. More precisely,
we associate to $\Delta$ a lattice polytope $\Delta' \subseteq \RR^{d+1}$ by considering for each
$(a,b) \in \Delta \cap \ZZ^2$ the lattice point 
\[ (a, 0, 0, \dots, 1, \dots, 0), \qquad \text{where the $1$ is in the $(b+1)$st place (omitted if $b=0$),} \]
and taking the convex hull. For example:
\begin{center}
 \begin{tikzpicture}[scale=0.5]
   \draw [thick] (0,0) -- (6,0) -- (7,1) -- (5,2) -- (1,2) -- (0,1) -- (0,0); 
   \node at (-0.5,-0.5) {\small $(0,0)$};   
   \node at (6.5,-0.5) {\small $(6,0)$};   
   \node at (7.9,1) {\small $(7,1)$};   
   \node at (5.5,2.5) {\small $(5,2)$};   
   \node at (0.5,2.5) {\small $(1,2)$};   
   \node at (-0.9,1) {\small $(0,1)$};   
   \node at (3.4,1) {\small $\Delta$};   
 \end{tikzpicture}
 \qquad \qquad \quad
  \begin{tikzpicture}[scale=0.6]
   \draw [dashed] (0,0) -- (2,1.5);  
   \draw [dashed] (6,0) -- (6,1.5);  
   \draw [thick] (0,0) -- (6,0) -- (7,1) -- (0,1) -- (0,0); 
   \draw [thick] (7,1) -- (6,1.5) -- (2,1.5) -- (0,1); 
   \node at (-0.5,-0.5) {\small $(0,0,0)$};   
   \node at (6.5,-0.5) {\small $(6,0,0)$};   
   \node at (8.1,1) {\small $(7,1,0)$};   
   \node at (-1.1,1) {\small $(0,1,0)$};       
   \node at (5.5,1.9) {\small $(5,0,1)$};   
   \node at (1.5,1.9) {\small $(1,0,1)$};    
   \node at (3.4,0.5) {\small $\Delta'$};   
 \end{tikzpicture}
\end{center}
Then our scroll is just the toric variety $X_{\Delta'}$ associated to $\Delta'$; this is unambiguously defined because $\Delta'$ is normal, as is easily seen using \cite[Prop.\,1.2.2]{bruns}. We denote
its defining ideal viewed inside $I_\Delta \subseteq S_\Delta$ by $I_{\Delta'}$.

As a generalization of \eqref{toricbettileeginwendige}, it is known that a minimal free resolution of the coordinate ring $S_\Delta / I_{\Delta'}$ of a rational normal scroll
is given by the Eagon--Northcott complex, from which it follows that the graded Betti table of $X_{\Delta'}$ has the following shape:
\begin{equation} \label{bettieagonnorthcott}
\small
\begin{array}{r|ccccccc}
  & 0 & 1 & 2 & 3 & \dots & f-2 & f-1 \\
\hline
0 & 1 & 0 & 0 & 0 & \dots & 0 & 0 \\
1 & 0 & {f \choose 2} & 2 {f \choose 3} & 3 {f \choose 4} & \dots & (f-2) {f \choose f-1} & (f-1) {f \choose f} \\
\end{array}
\end{equation}
where $f = \deg X_{\Delta'} = N_{\Delta'} - d - 1 = N_\Delta - d - 1$. Because all syzygies are linear, this must be a summand of the graded Betti table of $X_{\Delta}$, from which it follows that:
\begin{lemma}\label{Kp1upperbound}
$\min \{ \, \ell \, | \, b_{N_\Delta-\ell} \neq 0 \, \} \leq \lw(\Delta) + 2$.
\end{lemma}

\subsection{Explicit construction of non-exact cycles}

We can give an alternative proof of Lemma~\ref{Kp1upperbound} by explicitly constructing non-zero elements
in Koszul cohomology. From a geometric point of view this approach is less enlightening, but it 
allows us to prove the sharper bound $\min \{ \, \ell \, | \, b_{N_\Delta-\ell} \neq 0 \, \} \leq \lw(\Delta) + 1$
in the cases $\Delta \cong d\Sigma, \Upsilon_d$ ($d \geq 2$) and $\Delta \cong 2\Upsilon$. As we will see,
the sharper bound for $d\Sigma$ immediately implies the sharper bound for $\Upsilon_d$.

For $\ell = 1, \dots, N_\Delta - 3$ recall that $b_\ell$ is the cohomology in the middle of
\[
    \wedgepow{\ell+1} V_\Delta
    \stackrel{\delta}{\longrightarrow}
    \wedgepow{\ell} V_\Delta \otimes V_{\Delta} 
    \stackrel{\delta'}{\longrightarrow}
    \wedgepow{\ell-1} V_\Delta \otimes V_{2\Delta}.
\]
It is convenient to view this as a subcomplex of
\[
    \wedgepow{\ell+1} V_\Delta \otimes V_{\ZZ^2}
    \stackrel{\delta_{\ZZ^2}}{\longrightarrow}
    \wedgepow{\ell} V_\Delta \otimes V_{\ZZ^2} 
    \stackrel{\delta'_{\ZZ^2}}{\longrightarrow}
    \wedgepow{\ell-1} V_\Delta \otimes V_{\ZZ^2},
\]
where $V_{\ZZ^2} = k[x^{\pm 1}, y^{\pm 1}]$.
In what follows we will abuse notation and 
describe the basis elements of $V_\Delta$ and $V_{\ZZ^2}$ using the points
$(i,j) \in \ZZ^2$ rather than the monomials $x^iy^j$. 

Our technique to construct an element of $\ker \delta' \setminus \im \delta$ will be to apply $\delta_{\ZZ^2}$ to an element of $\wedgepow{\ell+1} V_\Delta\otimes V_{\ZZ^2}$ such that the result is in
$\wedgepow{\ell} V_\Delta\otimes V_\Delta$. This result 
will then automatically be contained in $\ker \delta'$, but it might land outside $\im \delta$. 
We first state an easy lemma that will be helpful in proving that certain elements are indeed not contained in $\im \delta$. 
Fix a strict total order $<$ on $\Delta \cap \ZZ^2$ and consider the bases
$$B=\{P_1\wedge\ldots\wedge P_{\ell+1} \, | \, P_1<\ldots<P_{\ell+1},\, P_1,\ldots,P_{\ell+1}\in\Delta \cap \ZZ^2\},$$
$$B'=\{P_1\wedge\ldots\wedge P_\ell \otimes P \, | \, P_1<\ldots<P_\ell, \, P,P_1,\ldots,P_\ell \in \Delta \cap \ZZ^2 \}$$
of $\wedgepow{\ell+1} V_\Delta$ and $\wedgepow{\ell} V_\Delta\otimes V_\Delta$, respectively.
\begin{lemma} 
If $x\in\wedgepow{\ell +1} V_\Delta$ has $n$ non-zero coordinates with respect to $B$, then $\delta(x)$ has $(\ell+1)n$ non-zero coordinates with respect to $B'$. 
\end{lemma}

\begin{proof} 
Write $x=\sum_{i=1}^{n}a_iP_{i,1}\wedge\ldots\wedge P_{i, \ell +1}$, $a_i\in k \setminus \{0\}$, where the $P_{i,1}\wedge\ldots\wedge P_{i, \ell +1}$'s are distinct elements of $B$. Then
$$\delta(x)=\sum_{i=1}^{n}\sum_{j=1}^{\ell +1}(-1)^ja_iP_{i,1}\wedge\ldots\wedge\widehat{P_{i,j}}\wedge\ldots\wedge P_{i, \ell +1}\otimes P_{i,j}$$
Each term in this sum is $\pm a_i$ times an element of $B'$, and the number of terms is $(\ell+1)n$, so we just have to verify 
that these
elements of $B'$ are mutually distinct, but that is easily done. 
\end{proof}

Our alternative proof of the upper bound $\min \{ \, \ell \, | \, b_{N_\Delta-\ell} \neq 0 \, \} \leq \lw(\Delta) + 2$ now goes as follows.

\begin{proof}[Alternative proof of Lemma~\ref{Kp1upperbound}]
As before, we can assume that $\Delta\subseteq\RR\times[0,d]$ with $d = \lw(\Delta)$.
Let $\ell=N_\Delta-d-2$ and let $P_1,\ldots,P_{\ell+1}$ be the points $(i,j)\in\Delta$ for which $i>m_j$, indexed so that $P_1<\ldots<P_{\ell+1}$. Now consider
\begin{align}
y&=\delta_{\ZZ^2}(P_1\wedge\ldots\wedge P_{\ell+1}\otimes(-1,0))\nonumber \\&=\sum_{s=1}^{\ell+1}(-1)^sP_1\wedge\ldots\wedge\widehat{P_{s}}\wedge\ldots\wedge P_{\ell+1}\otimes(P_s+(-1,0)).\nonumber
\end{align}
Clearly $y \in \wedgepow{\ell} V_\Delta\otimes V_\Delta$ and therefore $y \in \ker \delta'$.
So it remains to show that $y\notin\im\delta$. Suppose $y=\delta(x)$ for some $x\in\wedgepow{\ell+1} V_\Delta$. 
Since $y$ has $\ell+1$ nonzero coordinates with respect to the basis $B'$, by the previous lemma $x$ has just one non-zero coordinate with respect to 
the basis $B$. Therefore we can write
$$x=aP_1'\wedge\ldots\wedge P_{\ell+1}', \quad a\in k \setminus \{0\}, \quad P_1'<\ldots<P_{\ell+1}',$$
so that
$$y=\delta(x)=\sum_{s=1}^{\ell+1}a(-1)^sP_1'\wedge\ldots\wedge\widehat{P_{s}'}\wedge\ldots\wedge P_{\ell+1}'\otimes P_s'.$$
Comparing both expressions for $y$, we deduce that $\{P_1,\ldots,P_{\ell+1}\}=\{P_1',\ldots,P_{\ell+1}'\}$. 
This gives us a contradiction since the two expressions for $y$ have a different bidegree. 
Summing up, we have shown that $b_{N_\Delta - d - 2} \neq 0$, from which Lemma~\ref{Kp1upperbound} follows.
\end{proof}

The same proof technique enables us to deduce a sharper bound in the 
exceptional cases $d\Sigma$ ($d \geq 2$) and $2\Upsilon$.

\begin{lemma} \label{dSigmaupperbound}
If $\Delta \cong d \Sigma$ for some $d \geq 2$ then $\min \{ \, \ell \, | \, b_{N_\Delta-\ell} \neq 0 \, \} \leq \lw(\Delta) + 1$.
\end{lemma}
\begin{proof}
We can of course assume that $\Delta = d\Sigma$. Recall that $N_\Delta = (d+1)(d+2)/2$ and that $\lw(\Delta) = \lw(d\Sigma) = d$. Let $\ell=N_\Delta-d-1 = d(d+1)/2$. Let $P_1,\ldots,P_\ell$ be the elements of $(d-1)\Sigma \cap \ZZ^2$
and define
\begin{align}
y&=\delta_{\ZZ^2}\Big((d-1,1)\wedge P_1\wedge\ldots\wedge P_\ell)\otimes (1,0)-(d,0)\wedge P_1\wedge\ldots\wedge P_\ell\otimes(0,1)\Big)\nonumber \\
&=\sum_{s=1}^\ell(-1)^s(d,0)\wedge P_1\wedge\ldots\wedge\widehat{P_s}\wedge\ldots\wedge P_\ell\otimes(P_s+(0,1))\nonumber \\
&\hphantom{=}-\sum_{s=1}^\ell(-1)^s(d-1,1)\wedge P_1\wedge\ldots\wedge\widehat{P_s}\wedge\ldots\wedge P_\ell\otimes(P_s+(1,0)).\nonumber
\end{align}
As in the previous proof, since $y\in\wedgepow{\ell} V_\Delta\otimes V_\Delta$ we have $y\in\ker \delta'$. The fact that $y\notin \im \delta$ follows from the fact that the number of nonzero coordinates with respect to $B'$ is $2\ell$. If $y$ were in the image, then by our lemma $2\ell$ should be divisible by $\ell+1$, hence $\ell \leq 2$. But $\ell=d(d+1)/2\geq 3$ because $d\geq 2$: contradiction,
and the lemma follows.
\end{proof}
\begin{lemma} \label{Upsilonupperbound}
If $\Delta \cong 2 \Upsilon$ then $\min \{ \, \ell \, | \, b_{N_\Delta-\ell} \neq 0 \, \} \leq \lw(\Delta) + 1$.
\end{lemma}
\begin{proof}
Here we can assume $\Delta = 2 \Upsilon$ and note that $N_\Delta = 10$ and $\lw(\Delta) = \lw(2\Upsilon) = 4$. 
With $\ell = N_\Delta - d - 1 = 5$, in exactly the same way as before we see that
\begin{align}
\delta_{\ZZ^2}\Big(&(1,0)\wedge(0,1)\wedge(0,0)\wedge(-1,-1)\wedge(-1,0)\wedge(0,-1)\otimes(-1,-1)\nonumber \\
&+(1,0)\wedge(0,1)\wedge(0,0)\wedge(-1,-1)\wedge(0,-1)\wedge(-2,-2)\otimes(0,1)\nonumber \\
&-(1,0)\wedge(0,1)\wedge(0,0)\wedge(-1,-1)\wedge(-1,0)\wedge(-2,-2)\otimes(1,0)\Big)\nonumber
\end{align}
\noindent is a non-zero cycle: it has $12 = 2(\ell +1)$ terms, so if it were in $\im \delta$, then any preimage should
have two terms, and
we leave it to the reader to verify that this again leads to a contradiction. Alternatively, the reader
can just look up the graded Betti table of $X_{2\Upsilon}$ in Appendix~\ref{appendix_data}. \end{proof}

\begin{lemma} \label{Upsilondupperbound}
If $\Delta \cong \Upsilon_d$ for some $d \geq 2$ then $\min \{ \, \ell \, | \, b_{N_\Delta-\ell} \neq 0 \, \} \leq \lw(\Delta) + 1$.
\end{lemma}
\begin{proof}
From the combinatorics of $\Upsilon_d$ it is clear that 
if one restricts to those equations of $X_{\Upsilon_d}$ not involving
$X_{-1,-1}$, one obtains a set of defining equations for $X_{d \Sigma}$.
Thus the linear strand of the graded Betti table of $X_{d\Sigma}$
is a summand of the linear strand of the graded Betti table of $X_\Delta$.
From Lemma~\ref{dSigmaupperbound} we conclude that 
\[ \min \{ \, \ell \, | \, b_{N_\Delta-\ell} \neq 0 \, \} \leq \min \{ \, \ell \, | \, b_{N_{d\Sigma}-\ell} \neq 0 \, \} + 1 \leq \lw(d\Sigma) + 2 = d + 2. \]
The lemma follows from the observation that $\lw(\Delta) = d + 1$.
\end{proof}

\subsection{Conclusion}

Summarizing the results in this section, we state:

\begin{theorem} \label{Kp1conjectureoneinequality}
If $\Delta \not \cong \Sigma, \Upsilon$ then one has $\min \{ \, \ell \, | \, b_{N_\Delta-\ell} \neq 0 \, \} \leq \lw(\Delta) + 2$. If
\[ \Delta \cong d\Sigma \text{ for some $d \geq 2$} \ \quad \ \text{or} \ \quad \ \Delta \cong \Upsilon_d \text{ for some $d \geq 2$} \ \quad \ \text{or} \ \quad \ \Delta \cong 2 \Upsilon \]
then moreover one has the sharper bound $\lw(\Delta) + 1$. In other words the sharpest applicable upper bound predicted
by Conjecture~\ref{Kp1conjecture} holds.
\end{theorem}

\section{Pruning off vertices without changing the lattice width} \label{section_deformation}

\begin{theorem} \label{deformthm}
 Let $\Delta$ be a two-dimensional lattice polygon and let $p \geq 1$. Let $P$ be a vertex of $\Delta$ and define $\Delta' = \conv ( \Delta \cap \ZZ^2 \setminus \{P\} )$, where we assume that $\Delta'$ is two-dimensional. If
$K_{p,1}(X_{\Delta'}, L_{\Delta'}) = 0$ then also 
$K_{p+1,1}(X_{\Delta}, L_{\Delta}) = 0$.
\end{theorem}

\begin{proof} Consider
  $$\wedgepow{p+1} V_{\Delta'} \overset{\delta_1}{\longrightarrow}\wedgepow{p} V_{\Delta'} \otimes V_{\Delta'} \overset{\delta_2}{\longrightarrow} \wedgepow{p-1} V_{\Delta'} \otimes V_{2\Delta'} $$
 and
 $$\wedgepow{p+2} V_{\Delta} \overset{\delta_3}{\longrightarrow}\wedgepow{p+1} V_{\Delta} \otimes V_{\Delta} \overset{\delta_4}{\longrightarrow} \wedgepow{p} V_{\Delta} \otimes V_{2\Delta} $$
where the $\delta_i$'s are the usual coboundary maps.
Assuming that $\ker \delta_2 = \im \delta_1$ we will show that $\ker \delta_4 = \im \delta_3$. Suppose the contrary:
 we will find a contradiction.
Let $L:\RR^n\rightarrow\RR$ be a linear form that maps different lattice points in $\Delta$ to different numbers, such that $P$ attains the maximum of $L$ on $\Delta$. This exists because $P$ is a vertex. For any $x\in\wedgepow{p+1} V_{\Delta} \otimes V_{\Delta}$ define its support as the convex hull of the set of $P_{j,i}$'s occurring when expanding $x$ in the form
$$x=\sum_i\lambda_i P_{1,i}\wedge\ldots\wedge P_{p+1,i}\otimes Q_i.$$
Here as in Section~\ref{section_Kp1} we take the notational freedom to write points rather than monomials, and we of course assume that the elementary tensors in the above expression are mutually distinct.
Choose an $x\in\ker\delta_4\setminus \im\delta_3$ such that the maximum that $L$ attains on the support of $x$ is minimal, and let $P' \in \Delta \cap \ZZ^2$ be the unique point attaining this maximum. Rearrange the above expansion as follows:
\begin{equation}\label{writex}
x=\sum_i\lambda_iP'\wedge P_{1,i}\wedge\ldots\wedge P_{p,i}\otimes Q_i+\text{ terms not containing }P'\text{ in the }\wedge\text{ part}
\end{equation}
where all $P_{j,i}$'s are in $\Delta'$ and $Q_i\in \Delta$. We claim that in fact $Q_i \in \Delta'$, i.e.\ none of the $Q_i$'s equals $P$. Indeed, otherwise when applying $\delta_4$ the term $-\lambda_i P_{1,i}\wedge\ldots\wedge P_{p,i}\otimes(P'+Q_i)$ of $\delta_4(x)$ has nothing to cancel against, contradicting that $\delta_4(x)=0$. Let
\begin{equation}\label{definey}
y=\sum_i\lambda_iP_{1,i}\wedge\ldots\wedge P_{p,i}\otimes Q_i \in \wedgepow{p}V_{\Delta'} \otimes V_{\Delta'}.
\end{equation}
We have
$$0=\delta_4(x)=-P'\wedge\delta_2(y)+\text{ terms not containing }P'\text{ in the }\wedge\text{ part}.$$
Because terms of $P'\wedge\delta_2(y)$ cannot cancel against terms without $P'$ in the $\wedge$ part, $\delta_2(y)$ must be zero, and therefore $y\in\im \delta_1 $ by the exactness assumption. So write $y=\delta_1(z)$ with
$$z=\sum_i\mu_iP'_{1,i}\wedge\ldots\wedge P'_{p+1,i} \in \wedgepow{p+1}V_{\Delta'}.$$
Let $P''$ be the point occurring in this expression such that $L(P'')$ is maximal. Since there is no cancellation when applying $\delta_1$ one sees that $P''$ is in the support of $y$, hence in the support of $x$ and therefore $L(P'')<L(P')$. This means that $L$ achieves a smaller maximum on the support of $z$ than on the support of $x$. Finally, let
$$x'=x+\delta_3(P'\wedge z)=x-P'\wedge y-z\otimes P'.$$
Since $x\in \ker \delta_4 \setminus \im \delta_3$ we have $x' \in \ker \delta_4 \setminus \im \delta_3$ and by (\ref{writex}) and (\ref{definey}) one concludes that $L$ will achieve a smaller maximum on the support of $x'$ than on the support of $x$, namely $L(P'')$. This contradicts the choice of $x$.
\end{proof}

This immediately implies the following corollary, which is included in the statement of Theorem~\ref{theoremKp1progress} in the introduction.

\begin{corollary} \label{deformcor}
 Let $\Delta$ and $\Delta'$ be as in the statement of the above theorem. Assume that
 $\lw(\Delta) = \lw(\Delta')$, that $\Delta' \not \cong d \Sigma, \Upsilon_d$ for any $d \geq 1$ and that $\Delta' \not \cong 2\Upsilon$. If Conjecture~\ref{Kp1conjecture} holds for $\Delta'$ then it also
 holds for $\Delta$.
\end{corollary}

In order to deduce Conjecture~\ref{Kp1conjecture} for polygons having a small lattice width, we note the following.

\begin{lemma}
 Let $\Delta$ be a two-dimensional lattice polygon, let $d = \lw(\Delta)$, and assume 
 that removing an extremal lattice point makes the lattice width decrease, i.e.\
 for every vertex $P \in \Delta$ it holds that 
 \[ \lw(\conv ( \Delta \cap \ZZ^2 \setminus \{P\})) < d. \]
 Then there exists a unimodular transformation mapping $\Delta$ into $[0,d] \times [0,d]$.
\end{lemma}

\begin{proof}
 The cases where $\Delta^{(1)} \cong \emptyset$ or where $\Delta^{(1)} \cong d \Sigma$ for some $d \geq 0$ are easy to verify.
 In the other cases $\lw(\Delta^{(1)}) = \lw(\Delta) - 2 = d - 2$ and the lattice width directions for $\Delta$ and $\Delta^{(1)}$ are the same~\cite[Thm.\,13]{lubbesschicho}.
 Assume that $\Delta \subseteq \RR \times [0,d]$, fix a vertex on height $0$ and a vertex on height $d$, and let $P$ be any other vertex. 
 Then $\lw(\conv ( \Delta \cap \ZZ^2 \setminus \{P\})) \leq d-1$, 
 where we note that a corresponding lattice width direction is necessarily non-horizontal, and that along such a direction the width of $\Delta^{(1)}$ is at most $d-2$. But then equality must hold, and in particular it must also
 concern a lattice width direction for $\Delta^{(1)}$, hence it must concern a lattice width direction for $\Delta$. We conclude
 that $\Delta$ has two independent lattice width directions, and the lemma follows from the remark following~\cite[Lem.\,5.2]{linearpencils}.
\end{proof}

Let us call a lattice polygon $\Delta$ as in the statement of the foregoing lemma `minimal',
and note that this attribute applies to each of the exceptional polygons $d\Sigma, \Upsilon_d, 2\Upsilon$ mentioned in the statement of Conjecture~\ref{Kp1conjecture}.
In order to prove Conjecture~\ref{Kp1conjecture} for a certain non-exceptional polygon $\Delta$, by Corollary~\ref{deformcor}
it suffices to do this for any lattice polygon obtained by repeatedly pruning off vertices without changing the lattice width.
Thus the proof reduces to verifying the case of a minimal lattice polygon, unless it concerns one of
the exceptional cases $d\Sigma, \Upsilon_d, 2\Upsilon$, in which case one needs to stop pruning one step earlier (otherwise this strategy has no
chance of being successful).

In other words the above lemma implies that if Conjecture~\ref{Kp1conjecture} is true for all lattice polygons $\Delta$ for which $N_\Delta \leq (d+1)^2 + 1$,
then it is true for all lattice polygons $\Delta$ with $\lw(\Delta) \leq d$.
This observation, along with our exhaustive verification in the cases where $N_\Delta \leq 32$, reported upon in Section~\ref{section_computing}, allows us to conclude
that Conjecture~\ref{Kp1conjecture} is true as soon as $\lw(\Delta) \leq 4$. This fact will be used in the proof of
our explicit formula for $b_{N_\Delta - 4}$. 

But one can do better: in a spin-off paper~\cite{coolslemmens} devoted to minimal polygons, the second and the fourth author show
that if $\Delta$ is a minimal lattice polygon with $\lw(\Delta) \leq d$ then
\[ N_\Delta \leq \max \left\{ (d-1)^2 + 4, (d+1)(d+2)/2 \right\}. \]
From this, using a similar reasoning, the conjecture follows for $\lw(\Delta) \leq 6$, as announced in the statement of Theorem~\ref{theoremKp1progress}.

\section{Explicit formula for $b_{N_\Delta-4}$} \label{section_formulabN4}

In this section we will prove Theorem~\ref{thirdentry}, whose statement distinguishes between the following four cases:
\[ \left\{  \begin{array}{l} \Delta^{(1)} = \emptyset, \\ \dim \Delta^{(1)} = 0, \\ 
\dim \Delta^{(1)} = 1 \text{ or } \Delta \cong \Upsilon_2, \\
\dim \Delta^{(1)} = 2 \text{ and } \Delta\not \cong\Upsilon_2. \\
 \end{array} \right. \]
We will treat these cases in the above order, which as we will see corresponds to increasing order of difficulty.
The first case where $\Delta^{(1)} = \emptyset$ follows trivially from \eqref{toricbettileeginwendige}, so we can skip it.
Now recall
from (\ref{cohombelldual}) that $b_{N_\Delta-4}$ is the dimension of the cohomology in the middle of
\[
    \wedgepow{2} V_\Delta \otimes V_{\Delta^{(1)}} 
    \stackrel{\delta}{\longrightarrow}
    V_\Delta \otimes V_{(2\Delta)^{(1)}}
    \stackrel{\delta'}{\longrightarrow}
    V_{(3\Delta)^{(1)}}.
\]
Because $K_{0,3}(X;K,L) \cong K_{N_\Delta - 3,0}(X,L) = 0$, where we use that $\Delta \not \cong \Sigma$, we have that the map $\delta'$ is surjective.
In particular we obtain the formula
\[ b_{N_\Delta-4} = \dim \coker \delta - | (3\Delta)^{(1)}\cap\ZZ^2|. \]

\subsubsection*{Case $\dim \Delta^{(1)} = 0$}

If $\dim \Delta^{(1)}=0$ then $\delta$ is injective, so 
\[ b_{N_\Delta-4}= \dim (V_\Delta \otimes V_{(2\Delta)^{(1)}})- \dim(\wedgepow{2} V_\Delta)- |(3\Delta)^{(1)}\cap\ZZ^2| =(N_{\Delta}-4)(N_\Delta-1)/2, \] 
as can be calculated using Pick's theorem, thereby yielding Theorem~\ref{thirdentry} in this case (alternatively, one
can give an exhaustive proof by explicitly computing the graded Betti tables of the toric surfaces associated to the $16$ reflexive lattice polygons).

\subsubsection*{Case $\dim \Delta^{(1)} = 1$ or $\Delta \cong\Upsilon_2$}

The graded Betti table of $X_{\Upsilon_2}$ can be found in Appendix~\ref{appendix_data}, where one verifies
that $b_{N_{\Upsilon_2} - 4} = b_3 = 3$, as indeed predicted by the statement of Theorem~\ref{thirdentry}.
Therefore we can assume that $\dim \Delta^{(1)}=1$. The polygons
$\Delta$ having a one-dimensional interior were explicitly classified by Koelman~\cite[\S4.3]{koelman}, but in any case it
is easy to see that, using a unimodular transformation if needed, we can assume that 
\[ \Delta= \conv \{(m_1,1),(M_1,1),(m_0,0),(M_0,0),(m_{-1},-1),(M_{-1},-1)\} \] 
for some $m_i \leq M_i\in\ZZ$. Here $m_0 < M_0$ can be taken such that
\[ \Delta \cap (\ZZ \times \{0\}) = \{ m_0, m_0 + 1, \dots, M_0 \} \times \{0 \}. \]
Write $\Delta^{(1)}=[u,v]\times\{0\}$, then 
\[
 (2\Delta)^{(1)}=\Delta+\Delta^{(1)}=\conv\{ (m_i+u,i), (M_i+ v,i) \, | \, i=1,0,-1 \}.
\]
Now consider $V_\mathbb{Z} = k[x^{\pm 1}]$ and define a morphism
\[ f:V_\Delta\otimes V_{(2\Delta)^{(1)}}\rightarrow k[x_{-1}, x_0, x_1] \otimes V_\ZZ \] 
by letting 
$(a,b)\otimes(c,d) \mapsto x_b x_d \otimes(a+c)$, where again we abusingly describe
the basis elements of $V_\Delta$, $V_{(2\Delta)^{(1)}}$ and $V_\ZZ$ using lattice points
rather than monomials.
Note that 
\[ f( \delta ((a,b)\wedge(c,d)\otimes(e,0)))=f((a,b)\otimes(c+e,d)-(c,d)\otimes(a+e,b))=0, \]
so $\im \delta \subseteq \ker f$. 

We claim that actually equality holds. First note that every element $\alpha \in \ker f$
decomposes into elements
\[ \sum_j\lambda_j(a_j,b_j)\otimes(c_j,d_j) \] 
for which $(\{b_j,d_j\},a_j+c_j)$ is the same for all $j$: indeed, terms for which these are different cannot cancel out when applying $f$. Note that $\sum_j\lambda_j=0$, so
one can rewrite the above as a linear combination of expressions either of the form 
\[ \begin{array}{ccc} \underbrace{(a,b)\otimes(c,d)-(a',b)\otimes(c',d)} & \text{or of the form} & \underbrace{(a,b)\otimes(c,d)-(a',d)\otimes(c',b)} \\ \text{\small (i)} & & \text{\small (ii)} \\ \end{array} \] 
where $a+c=a'+c'$, the points $(a,b), (a',b)$ resp.\ $(a,b),(a',d)$ are in $\Delta$, and the points
$(c,d), (c',d)$ resp.\ $(c,d),(c',b)$ are in $(2\Delta)^{(1)}$. As for case (i), these can be decomposed further as a sum (or minus a sum) of
expressions of the form $(a,b)\otimes(c,d)-(a+1,b)\otimes(c-1,d)$, which can be rewritten as 
\[ \delta((a,b)\wedge(c-e,d)\otimes(e,0)-(a+1,b)\wedge(c-e,d)\otimes(e-1,0)) \]
and therefore as an element of $\im \delta$, at least if $e$ can be chosen in the interval $[ \max(u+1,c-M_d), \min(v,c-m_d)]$. The reader can verify that
this is indeed non-empty, from which the claim follows in this case. As for (ii), with $e$ chosen from the
non-empty interval $[\max(u,c'-M_b), \min(v,c'-m_b)]$ one verifies that
$$\delta((c'-e,b)\wedge(a',d)\otimes(e,0))=(c'-e,b)\otimes(a'+e,d)-(a',d)\otimes(c',b),$$
allowing one to replace (ii) with an expression of type (i), and the claim again follows.

Summing up, we have
\begin{align*}
b_{N_\Delta-4}=& \dim \im f - | (3\Delta)^{(1)}\cap\ZZ^2| \\
=&\sum_{\{i,j\}\subseteq\{-1,0,1\}} \left| [m_i+m_j+u,M_i+M_j+v] \cap \ZZ \right| -\sum_{i'=-2}^2 \left| (3\Delta)^{(1)}\cap(\ZZ\times\{i'\}) \right|.
\end{align*}
Each lattice point of $(3 \Delta)^{(1)} = 2\Delta + \Delta^{(1)}$ appears in an interval on the left, and conversely. To see
this it suffices to note that each lattice point of $2 \Delta$ arises as the sum of two lattice points in $\Delta$, which is a well-known property~\cite{HNPS}.
 So all terms with $i+j\neq 0$ cancel out the terms with $i' \neq 0$, and we are left with
\begin{multline*} \left| [m_1+m_{-1}+u,M_1+M_{-1}+v] \cap \ZZ \right| +  \left| [2m_0+u,2M_0+v] \cap \ZZ \right|  \\
 - \left| (3\Delta)^{(1)}\cap(\ZZ\times\{0\} )\right|.
\end{multline*}
Term by term this equals
\begin{multline*} \left( | \partial \Delta \cap \ZZ^2 | + N_{\Delta^{(1)}} - 2 - \varepsilon  \right) + \left( 2(M_0 - m_0) + N_{\Delta^{(1)}} \right) \\ - \left( 2(M_0 - m_0) + (2-\varepsilon) + N_{\Delta^{(1)}}  \right) \end{multline*}
where $\varepsilon := (u-m_0) + (M_0 - v) \in \{0,1,2\}$ denotes the cardinality of $\partial \Delta \cap (\ZZ \times \{0\})$.
Because the above expression simplifies to $N_{\Delta}-4$, this concludes the proof in the $\dim \Delta^{(1)} = 1$ case.

\subsubsection*{Case $\dim \Delta^{(1)} = 2$ and $\Delta\not \cong\Upsilon_2$}
In this case our task amounts to proving that $b_{N_\Delta - 4} = 0$, but
this follows from Conjecture~\ref{Kp1conjecture} for polygons $\Delta$ satisfying
$\lw(\Delta) \leq 4$, which was verified in Section~\ref{section_deformation}.

\section{Quotienting the Koszul complex} \label{section_alexandertruck}

We now start working towards an algorithmic determination of the graded Betti table of the toric surface $X_\Delta \subseteq \PP^{N_\Delta - 1}$ associated
to a given two-dimensional lattice polygon $\Delta$. Essentially, the method is about reducing the dimensions of the vector spaces involved, in order to make the linear
algebra more manageable. This is mainly done by incorporating bigrading and duality.
However when dealing with large polygons a further reduction is useful.
In this section we show that the Koszul complex always admits certain exact subcomplexes that can be described in a combinatorial way.
Quotienting out such a subcomplex does not affect the cohomology, while
making the linear algebra easier, at least in theory. 
For reasons we don't understand our practical implementation shows that the actual gain in runtime is somewhat unpredictable:
sometimes it is helpful, but other times the contrary is true. But it is worth the try, and
in any case we believe that the material below is also interesting from a theoretical point of view.

We first introduce the subcomplex from an algebraic point of view, then reinterpret things combinatorially, and finally specify our discussion
to the case of the Veronese surfaces $X_{d\Sigma}$. In the latter setting
the idea of quotienting out such an exact subcomplex is not new: for instance it appears in the recent paper by Ein, Erman and 
Lazarsfeld~\cite[p.\,2]{EinErmanLazarsfeld}.

\subsection{An exact subcomplex}\label{section_quotient}
We begin with the following lemma, which should be known to specialists, but we include a proof for the reader's convenience.
\begin{lemma}\label{quotient}
Let $M$ be a graded module over $k[x_1,\ldots,x_N]$ and suppose that the multiplication-by-$x_N$ map $M \rightarrow M$ is an injection. 
Then the Koszul complexes 
$$\ldots\rightarrow\wedgepow{p+1} V\otimes M \rightarrow \wedgepow{p} V\otimes M\rightarrow\wedgepow{p-1} V\otimes M\rightarrow\ldots$$
and 
$$\ldots\rightarrow\wedgepow{p+1} W\otimes M/(x_N M)\rightarrow\wedgepow{p} W\otimes M/(x_NM)\rightarrow\wedgepow{p-1} W\otimes M/(x_NM)\rightarrow\ldots$$
have the same graded cohomology. Here $V$ and $W$ denote the degree one parts of the polynomial rings $k[x_1,\ldots,x_N]$ and $k[x_1,\ldots,x_{N-1}]$, respectively.
\end{lemma}
\begin{proof}
Denote by $M'$ the graded module $M/(x_NM)$. For every $p\geq 0$ we have a short exact sequence
$$0\longrightarrow \left( \wedgepow{p} W\otimes M \right) \oplus \left( \wedgepow{p-1} W\otimes M \right) \overset{\alpha}{\longrightarrow}\wedgepow{p} V\otimes M\overset{\beta}{\longrightarrow}\wedgepow{p} W\otimes M'\longrightarrow 0,$$
by letting
\begin{multline*}
\alpha \left( \, v_1\wedge\ldots\wedge v_p\otimes m \, , \, w_1\wedge\ldots\wedge w_{p-1}\otimes m' \, \right) \\
= \, v_1\wedge\ldots\wedge v_p\otimes x_Nm \, + \, x_N\wedge w_1\wedge\ldots\wedge w_{p-1}\otimes m'
\end{multline*}
and $\beta(v_1\wedge\ldots\wedge v_p\otimes m)=\pi(v_1)\wedge\ldots\wedge\pi(v_p)\otimes \overline{m}$,
where $\pi:V\rightarrow W$ maps $x_i$ to itself if $i\neq N$ and to zero otherwise, and
$\overline{m}$ denotes the residue class of $m$ modulo $x_N M$. As usual if $p=0$ then it is understood that $\wedgepow{p-1} W\otimes M=0$. 
We leave a verification of the exactness to the reader, but note that the injectivity of the multiplication-by-$x_N$ map is important here.

On the other hand the spaces
\[ C_p = \left( \wedgepow{p} W\otimes M \right) \oplus \left( \wedgepow{p-1} W\otimes M \right) \]
naturally form a long exact sequence 
$\ldots\rightarrow C_2\rightarrow C_1\rightarrow C_0\rightarrow 0$ along the morphisms 
\[ d_p : C_p \rightarrow C_{p-1} : (a,b) \mapsto (-b+\delta_p(a),-\delta_{p-1}(b)) \]
where $\delta_p$ and $\delta_{p-1}$ are the usual coboundary maps, as described in \eqref{delta}. 
Exactness holds because if $d_p(a,b)=0$ then $d_{p+1}(0,-a)=(a,b)$.
Overall we end up with a short exact sequence of complexes:\footnotesize
\[\begin{array}{ccccccccc}
 & & \vdots   & & \vdots  & &  \vdots & & \\
 & & \downarrow & & \downarrow & & \downarrow & & \\
0 &\rightarrow &\wedgepow{p+1} W\otimes M\oplus\wedgepow{p} W\otimes M&\rightarrow&\wedgepow{p+1} V\otimes M&\rightarrow&\wedgepow{p+1} W\otimes M'&\rightarrow& 0 \\
 & & \downarrow & & \downarrow & & \downarrow & & \\
0 &\rightarrow &\wedgepow{p} W\otimes M\oplus\wedgepow{p-1} W\otimes M&\rightarrow&\wedgepow{p} V\otimes M&\rightarrow&\wedgepow{p} W\otimes M'&\rightarrow& 0 \\
 & & \downarrow & & \downarrow & & \downarrow & & \\
 & & \vdots & & \vdots & & \vdots & &
\end{array}\]\normalsize
This gives a long exact sequence in (co)homology, and the result follows from the exactness of the left column.
\end{proof}

Now we explain how to exploit the above lemma for our purposes. 
We can apply it to the Koszul complex 
$$\ldots\rightarrow\wedgepow{p+1} V_\Delta\otimes\bigoplus_{i\geq 0}V_{i\Delta}\rightarrow\wedgepow{p} V_\Delta\otimes\bigoplus_{i\geq 0}V_{i\Delta}\rightarrow\wedgepow{p-1} V_\Delta\otimes\bigoplus_{i\geq 0}V_{i\Delta}\rightarrow\ldots$$
as well as to the twisted Koszul complex
$$\ldots\rightarrow\wedgepow{p+1} V_\Delta\otimes\bigoplus_{i\geq 1}V_{(i\Delta)^{(1)}}\rightarrow\wedgepow{p} V_\Delta\otimes\bigoplus_{i\geq 1}V_{(i\Delta)^{(1)}}\rightarrow\wedgepow{p-1} V_\Delta\otimes\bigoplus_{i\geq 1}V_{(i\Delta)^{(1)}}\rightarrow\ldots$$
These are complexes of graded modules over the polynomial ring whose variables correspond to the lattice points of $\Delta$. 
In both cases the variable corresponding to whatever point $P \in \Delta \cap \ZZ^2$ can be chosen as $x_N$, because 
multiplication by $x_N$ will always be injective. 
Then the lemma yields that we can replace 
$V_{i\Delta}$ by $V_{(i\Delta)\backslash((i-1)\Delta+P)}$ in the first complex, and that we can replace 
$V_{(i\Delta)^{(1)}}$ by $V_{(i\Delta)^{(1)}\backslash(((i-1)\Delta)^{(1)}+P)}$ in the second complex. 
In both cases we must also replace the $V_\Delta$'s in the wedge product by $V_{\Delta\backslash\{P\}}$. 
Splitting these complexes into their graded pieces we conclude that $K_{p,q}(X,L)$ can be computed as the cohomology in the middle of 
\begin{align}
\wedgepow{p+1} V_{\Delta\backslash\{P\}}\otimes V_{((q-1)\Delta)\backslash((q-2)\Delta+P)}&\longrightarrow\wedgepow{p} V_{\Delta\backslash\{P\}}\otimes V_{(q\Delta)\backslash((q-1)\Delta+P)}\nonumber \\
&\longrightarrow\wedgepow{p-1} V_{\Delta\backslash\{P\}}\otimes V_{((q+1)\Delta)\backslash(q\Delta+P)},\nonumber
\end{align}
and that the twisted Koszul cohomology spaces $K_{p,q}(X;K,L)$ can be computed as the cohomology in the middle of
\begin{align}
\wedgepow{p+1} V_{\Delta\backslash\{P\}}\otimes V_{((q-1)\Delta)^{(1)}\backslash((q-2)\Delta+P)^{(1)}}&\longrightarrow\wedgepow{p} V_{\Delta\backslash\{P\}}\otimes V_{(q\Delta)^{(1)}\backslash((q-1)\Delta+P)^{(1)}}\nonumber \\
&\longrightarrow\wedgepow{p-1} V_{\Delta\backslash\{P\}}\otimes V_{((q+1)\Delta)^{(1)}\backslash(q\Delta+P)^{(1)}}.\nonumber
\end{align}
Here for any $A \subseteq \mathbb{Z}^2$ we let $V_A \subseteq k[x^{\pm 1}, y^{\pm 1}]$ denote the space of Laurent polynomials whose support is contained in $A$.

\begin{remark}\label{additionalrule}
The coboundary morphisms are still defined as in \eqref{delta},
with the additional rule that $x^i y^j$ is considered zero in $V_A$ as soon as $(i,j) \notin A$.
\end{remark}

\begin{remark} \label{quotientingrespectsgradingrmk}
 It is important to observe that the above complexes remain naturally bigraded, and that this is compatible with the bigrading described
 in Section~\ref{section_bigrading}. In other
 words, for any $(a,b) \in \ZZ^2$, also the spaces $K_{p,q}^{(a,b)}(X,L)$ and $K_{p,q}^{(a,b)}(X; K, L)$ can be computed from the above sequences.
\end{remark}

\subsection{Removing multiple points}
In some cases we can remove multiple points from $\Delta$ by applying Lemma~\ref{quotient} repeatedly.
In algebraic terms this works if and only if these points, when viewed as elements of $V_\Delta$, form a regular sequence
for the graded module $M$, where $M$ is either $\bigoplus_{i\geq 0}V_{i\Delta}$ or $\bigoplus_{i\geq 1}V_{(i\Delta)^{(1)}}$. The length of a regular sequence is bounded by the Krull dimension of $M$, which is equal to $3$. So we can never remove more than three points. It is well-known that for graded modules over Noetherian rings any permutation of a regular sequence is again a regular sequence, so the order of removing points does not matter. 
Concretely, after removing the points $P_1,\ldots,P_m$ we get the complex
\begin{multline*}
\ldots \longrightarrow\wedgepow{p+1} V_{\Delta\backslash\{P_1,\ldots,P_m\}}\otimes \frac{M_{q-1}}{P_1M_{q-2}+\ldots+P_mM_{q-2}} \\
 \longrightarrow\wedgepow{p} V_{\Delta\backslash\{P_1,\ldots,P_m\}}\otimes \frac{M_q}{P_1M_{q-1}+\ldots+P_mM_{q-1}}\longrightarrow\ldots
\end{multline*}
where $M_i$ denotes the degree $i$ part of $M$. Here, as before, we abuse notation and identify the points $P_i \in \Delta$ with the corresponding monomials in $V_\Delta$.
So for $M=\bigoplus_{i\geq 0}V_{i\Delta}$ this gives
\begin{multline*}
\ldots \longrightarrow\wedgepow{p+1} V_{\Delta\backslash\{P_1,\ldots,P_m\}}\otimes V_{(q-1)\Delta\backslash((P_1+(q-2)\Delta)\cup\ldots\cup(P_m+(q-2)\Delta))}\nonumber \\ \longrightarrow\wedgepow{p} V_{\Delta\backslash\{P_1,\ldots,P_m\}}\otimes V_{q\Delta\backslash((P_1+(q-1)\Delta)\cup\ldots\cup(P_m+(q-1)\Delta))}\longrightarrow\ldots
\end{multline*}
while for $M=\bigoplus_{i\geq 1}V_{(i\Delta)^{(1)}}$ it gives
\begin{multline}
\ldots \longrightarrow\wedgepow{p+1} V_{\Delta\backslash\{P_1,\ldots,P_m\}}\otimes V_{((q-1)\Delta)^{(1)}\backslash((P_1+((q-2)\Delta)^{(1)})\cup\ldots\cup(P_m+((q-2)\Delta)^{(1)}))}\nonumber \\ \longrightarrow\wedgepow{p} V_{\Delta\backslash\{P_1,\ldots,P_m\}}\otimes V_{q\Delta\backslash((P_1+((q-1)\Delta)^{(1)})\cup\ldots\cup(P_m+((q-1)\Delta)^{(1)}))}\longrightarrow\ldots
\end{multline}
The question we study in this section is which sequences of points $P_1, \dots, P_m \in \Delta \cap \ZZ^2$ are regular, where necessarily $m \leq 3$.

We first study the problem of which sequences of two points are regular. As for $M = \bigoplus_{i\geq 0}V_{i\Delta}$, if we first remove
a point $P \in \Delta \cap \ZZ^2$ then we end up with $M / PM$, whose graded components in degree $q \geq 1$ 
are of the form $V_{q\Delta\backslash(P+(q-1)\Delta)}$, while the degree 0 part is just $V_{0\Delta}$. Multiplication by another point $Q \in \Delta \cap \ZZ^2$ in $M/PM$ corresponds to
$$V_{q\Delta\backslash(P+(q-1)\Delta)}\overset{\cdot Q}{\longrightarrow}V_{(q+1)\Delta\backslash(P+q\Delta)}.$$
In order for the sequence $P,Q$ to be regular this map has to be injective for all $q \geq 1$. This means that
$$((q\Delta\backslash(P+(q-1)\Delta))+Q)\cap(P+q\Delta)\cap\ZZ^2=\emptyset.$$
Subtracting $P+Q$ yields
$$(q\Delta-P)\backslash((q-1)\Delta)\cap(q\Delta-Q)\cap\ZZ^2=\emptyset,$$
eventually leading to the criterion
\begin{multline} \label{xyregularcondition1}
\text{$P,Q$ is regular for $\bigoplus_{i\geq 0}V_{i\Delta}$} \quad \Leftrightarrow \\ \forall q\geq 1: (q\Delta-P)\cap(q\Delta-Q)\cap\ZZ^2\subseteq(q-1)\Delta.
\end{multline}
Similarly we find
\begin{multline} \label{xyregularcondition2}
\text{$P,Q$ is regular for $\bigoplus_{i\geq 1}V_{(i\Delta)^{(1)}}$} \quad \Leftrightarrow \\ \forall q\geq 1: (q\Delta-P)^{(1)}\cap(q\Delta-Q)^{(1)}\cap\ZZ^2\subseteq ((q-1)\Delta)^{(1)}.
\end{multline}
These criteria are strongly simplified by the equivalences $\emph{\ref{thm-remove-reg}}.\iff \emph{\ref{thm-remove-reg-int}}.
\iff \emph{\ref{thm-remove-quadrangle}}.$
of the following theorem:

\begin{theorem}\label{thm-remove}
Let $\Delta$ be a two-dimensional lattice polygon.
For two distinct lattice points $P, Q \in \Delta$, the following are equivalent:
\begin{enumerate}
\item\label{thm-remove-reg} $P,Q$ is a regular sequence for $\bigoplus_{i\geq 0}V_{i\Delta}$.
\item\label{thm-remove-reg-int} $P,Q$ is a regular sequence for $\bigoplus_{i\geq 1}V_{(i\Delta)^{(1)}}$.
\item\label{thm-remove-exists} $(q \Delta - P) \cap (q \Delta - Q) \subseteq (q-1) \Delta$ for some $q > 1$.
\item\label{thm-remove-forall} $(q \Delta - P) \cap (q \Delta - Q) \subseteq (q-1) \Delta$ for all $q \geq 1$.
\item\label{thm-remove-forall-int} $((q \Delta)\topinterior - P) \cap ((q \Delta)\topinterior - Q) \subseteq ((q-1) \Delta)\topinterior$ for all $q \geq 1$, where $\topinterior$ denotes the interior for the standard 
topology on $\RR^2$.
\item\label{thm-remove-lattice-int} $((q \Delta)\interior - P) \cap ((q \Delta)\interior - Q) \cap \mathbb{Z}^2 \subseteq ((q-1) \Delta)\interior \cap \mathbb{Z}^2$ for all $q \geq 1$.
\item\label{thm-remove-lattice} $(q \Delta - P) \cap (q \Delta - Q) \cap \mathbb{Z}^2 \subseteq (q-1) \Delta \cap \mathbb{Z}^2$ for all $q \geq 1$.
\item\label{thm-remove-halfplane} Let $\ell$ be the line through $P$ and $Q$.
For both half-planes $H$ bordered by $\ell$, the polygon $H \cap \Delta$
is a triangle with $P$ and $Q$ as two vertices
(this may be degenerate, in which case it is the line segment $P Q$).
\item\label{thm-remove-quadrangle} $\Delta$ is a quadrangle and $P$ and $Q$
are opposite vertices of this quadrangle (this may be the degenerate case where 
$\Delta$ is a triangle and $P, Q$ are any
pair of vertices of $\Delta$).
\end{enumerate}
\end{theorem}

\begin{proof}
The equivalences $\emph{\ref{thm-remove-reg}}.\iff\emph{\ref{thm-remove-lattice}}.$ and $\emph{\ref{thm-remove-reg-int}}.\iff\emph{\ref{thm-remove-lattice-int}}.$ follow from the foregoing discussion.\newline
$\emph{\ref{thm-remove-exists}}. \implies \emph{\ref{thm-remove-forall}}.$:
assume that \emph{\ref{thm-remove-exists}}.\ holds for some $q > 1$.
Let $q' \geq 1$, we show that it also holds for $q'$.
Let $W \in (q' \Delta - P) \cap (q' \Delta - Q)$,
we need to show that $W \in (q' - 1) \Delta$.

In case $q' > q$, we define $\delta = (q - 1)/(q' - 1) < 1$.
Now consider
\begin{align*}
    W &\in ((q'-1) \Delta + (\Delta - P)) \cap ((q'-1) \Delta + (\Delta - Q)) \\
    \delta W &\in ((q-1) \Delta + \delta (\Delta - P)) \cap ((q-1) \Delta + \delta (\Delta - Q)) \\
        &\subseteq ((q-1) \Delta + (\Delta - P)) \cap ((q-1) \Delta + (\Delta - Q)) \\
        &= (q \Delta - P) \cap (q \Delta - Q) \subseteq (q - 1) \Delta.
\end{align*}
We conclude that $W \in (q' - 1) \Delta$.

If $q' < q$, we find
\begin{align*}
    W + (q - q')\Delta &\subseteq \left[(q' \Delta - P) \cap (q' \Delta - Q)\right] + (q - q') \Delta \\
        &\subseteq (q' \Delta - P + (q - q') \Delta) \cap (q' \Delta - Q + (q - q') \Delta) \\
        &\subseteq (q \Delta - P) \cap (q \Delta - Q) \subseteq (q-1) \Delta.
\end{align*}
Since $W + (q - q') \Delta \subseteq (q-1) \Delta$, it follows that $W \in (q' - 1) \Delta$.

$\emph{\ref{thm-remove-forall}}. \implies \emph{\ref{thm-remove-forall-int}}.$:
this holds by taking interiors on both sides and using the fact that
$(A \cap B)\topinterior = A\topinterior \cap B\topinterior$.

$\emph{\ref{thm-remove-forall-int}}. \implies \emph{\ref{thm-remove-lattice-int}}.$:
intersect with $\mathbb{Z}^2$ on both sides and use $\Delta\topinterior \cap \mathbb{Z}^2 = \Delta\interior \cap \ZZ^2$.

$\emph{\ref{thm-remove-lattice-int}}. \implies \emph{\ref{thm-remove-lattice}}.$:
let $W \in (q \Delta - P) \cap (q \Delta - Q) \cap \mathbb{Z}^2$.
\begin{align*}
    W + \left((3\Delta)\interior \cap \mathbb{Z}^2\right)
        &= \left(W + (3\Delta)\interior\right) \cap \mathbb{Z}^2 \\
        &\subseteq \left[q \Delta + (3 \Delta)\interior - P\right] \cap \left[q \Delta + (3 \Delta)\interior - Q\right] \cap \mathbb{Z}^2 \\
        &\subseteq \left[((q+3) \Delta)\interior - P\right] \cap \left[((q+3) \Delta)\interior - Q\right] \cap \mathbb{Z}^2 \\
        &\subseteq ((q+2) \Delta)\interior \cap \mathbb{Z}^2.
\end{align*}
Since $(3\Delta)\interior$ must contain a lattice point,
it follows that $W \in (q-1) \Delta \cap \mathbb{Z}^2$.

$\emph{\ref{thm-remove-lattice}}. \implies \emph{\ref{thm-remove-halfplane}}.$:
we show this by contraposition, so we assume that item \emph{\ref{thm-remove-halfplane}}.\ is not satisfied
for a half-plane $H$. 

\begin{figure}[thb]
\begin{minipage}{0.45\textwidth}
\begin{tikzpicture}[scale=0.4]
\coordinate [label=below:${P=O}$] (P) at (0,0);
\coordinate [label=below:$Q$] (Q) at (10,0);
\coordinate [label=above:$T$] (T) at (6,6);
\coordinate [label=above:$R$] (R) at (2,4);
\coordinate (s) at (1,3);
\coordinate [label=above:$\ell$] (ell) at (12,0);

\filldraw[fill=black, fill opacity=0.1] (P) -- (Q) -- (T) -- (R) -- (s) -- (P);
\node at (5, 2) {$H \cap \Delta$};

\draw[help lines] (-2,0) -- (ell);
\draw[dashed] (-2,6) -- (12,6);
\draw[dotted] (P) -- (T);

\fill (P) circle(5pt);
\fill (Q) circle(5pt);
\fill (T) circle(5pt);
\fill (R) circle(5pt);
\end{tikzpicture}
\caption{$\ref{thm-remove-lattice}. \implies \ref{thm-remove-halfplane}.$}
\end{minipage}
\hfill
\begin{minipage}{0.45\textwidth}
\begin{tikzpicture}[scale=0.4]
\coordinate [label=below:${P=O}$] (P) at (0,0);
\coordinate [label=below:$Q$] (Q) at (5,0);
\coordinate [label=below:$T$] (T) at (8,0);
\coordinate [label=below:$R$] (R) at (-3,0);
\coordinate [label=above:$\ell$] (ell) at (10,0);

\draw (R) -- (T);

\draw[help lines] (-4,0) -- (ell);

\fill (P) circle(5pt);
\fill (Q) circle(5pt);
\fill (T) circle(5pt);
\fill (R) circle(5pt);
\end{tikzpicture}
\caption{degenerate case (where $T$ may be equal to $Q$)}
\end{minipage}
\end{figure}

Let $T$ a vertex of $H \cap \Delta$ at maximal distance from $\ell$, and assume for now that this
distance is positive. Let $R$ be a vertex of $H \cap \Delta$, distinct from $P$, $Q$ and $T$
(the fact that such an $R$ exists follows from the assumption).
Without loss of generality,
we may assume that $R$ lies in the half-plane bordered by the line $P T$ that does not contain $Q$.
Choose coordinates such that the origin is $P$.

Equip $R$ with barycentric coordinates
\begin{equation}\label{thm-remove-r-bary}
    R = \alpha T + \beta Q + \gamma P = \alpha T + \beta Q.
\end{equation}
Because of the position of $R$, we know that $0 \leq \alpha \leq 1$ and $\beta < 0$.

Choose an integer $q > \max\{1, -\beta^{-1}\}$.
Let $W = q R$.
We claim that
\begin{equation}
    W \in \left( (q \Delta) \cap (q \Delta - Q) \cap \mathbb{Z}^2 \right) \setminus (q-1) \Delta,
\end{equation}
contradicting \emph{\ref{thm-remove-lattice}}.
Since $R$ is a vertex of $H \cap \Delta$, we immediately have
$W \in \left((q \Delta) \cap \mathbb{Z}^2 \right) \setminus (q-1) \Delta$.
It remains to show that $W \in q \Delta - Q$.
Using \eqref{thm-remove-r-bary}, we have
\begin{align*}
    W + Q &= q R + Q = q R + \beta^{-1}(R - \alpha T) \\
        &= (q + \beta^{-1})R + (- \beta^{-1} \alpha) T
\end{align*}
This is a convex combination of $qP = O$, $q R$ and $q T$ because
\[
    q + \beta^{-1} \geq 0, \qquad - \beta^{-1} \alpha \geq 0,
\]
and
\[
    (q + \beta^{-1}) + ( - \beta^{-1} \alpha) = q + \beta^{-1}(1 - \alpha) \leq q.
\]
It follows that $W + Q \in q\Delta$.

In the degenerate case where $T \in \ell$,
without loss of generality one can assume that there is a vertex $R$ such that $R$ and $Q$ lie on opposite sides of $P = O$. One proceeds
as above with $\alpha = 0$ and $\beta < 0$.

$\emph{\ref{thm-remove-halfplane}}. \implies \emph{\ref{thm-remove-quadrangle}}.$:
this follows immediately from the geometry:
$\Delta$ must be the union of two triangles on the base $P Q$.

\begin{figure}
\begin{tikzpicture}[scale=0.25]
\coordinate [label=above:$R$] (R) at (6,4);
\coordinate [label=below:$S$] (S) at (5,-3);
\coordinate [label=left:$P$] (P1) at (0,0);
\coordinate [label=right:$2Q-P$] (Q1) at (20,0);
\coordinate [label=above:$2R-P$] (R1) at (12,8);
\coordinate [label=below:$2S-P$] (S1) at (10,-6);
\coordinate [label=left:$2P-Q$] (P2) at (-10,0);
\coordinate [label=right:$Q$] (Q2) at (10,0);
\coordinate [label=above:$2R-Q$] (R2) at (2,8);
\coordinate [label=below:$2S-Q$] (S2) at (0,-6);
\coordinate (W1) at (15,10);
\coordinate (Z1) at (12.5,-7.5);
\coordinate (W2) at (0,10);
\coordinate (Z2) at (-2.5,-7.5);

\draw[help lines] (W1) -- (P1) -- (Z1);
\draw[help lines] (W2) -- (Q2) -- (Z2);
\filldraw[fill=black, fill opacity=0.1] (P1) -- (R1) -- (Q1) -- (S1) --(P1);
\filldraw[fill=black, fill opacity=0.1] (P2) -- (R2) -- (Q2) -- (S2) --(P2);
\node at (5, 0) {$\Delta$};
\node at (0, 4) {$2 \Delta - Q$};
\node at (13, 4) {$2 \Delta - P$};

\fill (R) circle(8pt);
\fill (S) circle(8pt);
\fill (P1) circle(8pt);
\fill (Q1) circle(8pt);
\fill (R1) circle(8pt);
\fill (S1) circle(8pt);
\fill (P2) circle(8pt);
\fill (Q2) circle(8pt);
\fill (R2) circle(8pt);
\fill (S2) circle(8pt);
\end{tikzpicture}
\caption{$\emph{\ref{thm-remove-quadrangle}}. \implies \emph{\ref{thm-remove-exists}}.$ with $q=2$}
\end{figure}

$\emph{\ref{thm-remove-quadrangle}}. \implies \emph{\ref{thm-remove-exists}}.$:
we show this for $q = 2$.
By assumption, the lattice polygon $\Delta$ is a convex quadrangle $P R Q S$
(possibly degenerated into a triangle, i.e.\ one of $R$ or $S$ may coincide with $P$ or $Q$).
We need to show that
\begin{equation}
    (2 \Delta - P) \cap (2 \Delta - Q) \subseteq \Delta
\end{equation}
The left hand side is clearly contained in the cones
$\widehat{R P S}$ and $\widehat{R Q S}$, whose intersection is precisely 
our quadrangle $P R Q S = \Delta$.
\end{proof}

Now let us switch to regular sequences consisting of three points. We have the following easy fact:
\begin{lemma} Let $P,Q,R \in \Delta \cap \ZZ^2$ be distinct. Then
$P,Q,R$ is a regular sequence for $M=\bigoplus_{i\geq 0}V_{i\Delta}$ (resp.\ $M=\bigoplus_{i\geq 1}V_{(i\Delta)^{(1)}}$) if 
and only if 
\[ P,Q, \qquad Q,R, \qquad P,R \] 
are regular sequences.
\end{lemma}
\begin{proof} It is clearly sufficient to prove the `if' part of the claim. Assume for simplicity that $M = \bigoplus_{i\geq 0}V_{i\Delta}$, the other case is similar.
Since $P,Q$ is regular, all we have to check is that
$$V_{q\Delta\backslash((P+(q-1)\Delta)\cup(Q+(q-1)\Delta))}\overset{ \cdot R}{\longrightarrow} V_{(q+1)\Delta\backslash((P+q\Delta)\cup(Q+q\Delta))}$$
is injective, or equivalently that
$$\left( q\Delta\backslash((P+(q-1)\Delta)\cup(Q+(q-1)\Delta))+R \right) \cap((P+q\Delta)\cup(Q+q\Delta))=\emptyset.$$
This condition can be rewritten as
\begin{equation} \label{regularsequenceoflength3}
q\Delta\cap((q\Delta+P-R)\cup(q\Delta+Q-R))\subseteq(P+(q-1)\Delta)\cup(Q+(q-1)\Delta).
\end{equation}
Since $P,R$ is regular we know that $q\Delta\cap(q\Delta + P -R)\subseteq P + (q-1)\Delta$ by \eqref{xyregularcondition1}.
Similarly because $Q,R$ is regular we have $q\Delta\cap(q\Delta + Q -R)\subseteq Q + (q-1)\Delta$. Together 
these two inclusions imply \eqref{regularsequenceoflength3}.
\end{proof}

As an immediate corollary, we deduce using Theorem~\ref{thm-remove}:

\begin{corollary}\label{cor-remove-three}
Let $\Delta$ be a two-dimensional lattice polygon.
For three distinct lattice points $P, Q, R \in \Delta$, the following statements are equivalent:
\begin{enumerate}
\item\label{cor-remove-reg} $P,Q,R$ is a regular sequence for $\bigoplus_{i\geq 0}V_{i\Delta}$.
\item\label{cor-remove-reg-int} $P,Q,R$ is a regular sequence for $\bigoplus_{i\geq 1}V_{(i\Delta)^{(1)}}$.
\item\label{cor-remove-triangle} $\Delta$ is a triangle with vertices $P$, $Q$ and $R$.
\end{enumerate}
\end{corollary}

\subsection{Example: the case of Veronese embeddings}
Let us apply the foregoing to $\Delta=d\Sigma$ for $d \geq 2$, whose corresponding toric surface
is the Veronese surface $\nu_d(\PP^2)$ with coordinate ring
\begin{equation} \label{gradedcoordinateringdsigma} 
 S_{d\Sigma} \cong k \oplus V_{d\Sigma} \oplus V_{2d\Sigma} \oplus V_{3d\Sigma} \oplus V_{4d\Sigma} \oplus V_{5d\Sigma} \oplus \ldots 
\end{equation}
By the foregoing corollary the sequence of points $(0,d)$, $(d,0)$, $(0,0)$ is regular for $S_{d\Sigma}$.
When one removes these points along the above guidelines, the resulting graded module is
\[ k \oplus V_{d\Sigma \setminus \{(0,d),(d,0),(0,0)\}} \oplus V_{\conv\{(d-1,d-1),(2,d-1),(d-1,2)\}} \oplus 0 \oplus 0 \oplus 0 \oplus \ldots \]
 which can be rewritten as
\begin{equation} \label{gradedcoordinateringdsigmaremoved}
k \oplus V_{d\Sigma \setminus \{(0,d),(d,0),(0,0)\}} \oplus V_{(d,d) - (d\Sigma)^{(1)}} \oplus 0 \oplus 0 \oplus 0 \oplus \ldots
\end{equation}
We recall from the end of Section~\ref{section_quotient} that multiplication is defined by lattice addition, with the convention that the product is zero whenever the sum falls outside the indicated range. In order to find the graded Betti table of $\nu_d(\PP^2)$, it therefore suffices to compute the cohomology of complexes of the following type:
\begin{multline} \label{veronesecomplexdualremoved}
\wedgepow{\ell+1} V_{d\Sigma \setminus \{(0,d),(d,0),(0,0)\}}\longrightarrow\wedgepow{\ell} V_{d \Sigma \setminus \{(0,d),(d,0),(0,0)\}}\otimes V_{d \Sigma \setminus\{(0,d),(d,0),(0,0)\}}\\\longrightarrow\wedgepow{\ell-1} V_{d \Sigma \setminus \{(0,d),(d,0),(0,0)\}}\otimes V_{(d,d)-(d \Sigma)^{(1)}}
\end{multline}
Indeed, the cohomology in the middle has dimension $\dim K_{\ell,1}(X,L) = b_\ell$ and the cokernel of the second morphism has 
dimension $\dim K_{\ell-1,2}(X,L) = c_{N_\Delta - 1 - \ell}$. 

We can carry out the same procedure in the twisted case. The resulting graded module is 
\[ k \oplus V_{(d\Sigma)^{(1)}} \oplus V_{(d,d) - d\Sigma \setminus \{(0,d),(d,0),(0,0)\} } \oplus V_{\{(d,d)\}} \oplus 0 \oplus 0 \oplus \ldots \]
For instance, one finds that $K_{\ell,1}^\vee(X,L) \cong K_{N_\Delta-3-\ell,2}(X;K,L)$ is the cohomology in the middle of
\begin{multline*}
\wedgepow{N_\Delta-\ell-2} V_{(d\Sigma) \setminus \{(0,d),(d,0),(0,0)\}}\otimes V_{(d \Sigma)^{(1)}}\longrightarrow \\ \wedgepow{N_\Delta-\ell-3} V_{d \Sigma \backslash\{(0,d),(d,0),(0,0)\}}\otimes V_{(d,d)-d \Sigma \backslash\{(0,d),(d,0),(0,0)\}} \\ \longrightarrow\wedgepow{N_\Delta-\ell-4} V_{d \Sigma \backslash\{(0,d),(d,0),(0,0)\}}\otimes V_{(d,d)}.
\end{multline*}
As a side remark, note that this complex is isomorphic to the dual of (\ref{veronesecomplexdualremoved}). Thus this gives a combinatorial proof of the duality 
formula $K_{\ell,1}^\vee(X,L) \cong K_{N_\Delta-3-\ell,2}(X;K,L)$ for Veronese surfaces.


Let us conclude with a visualization of the point removal procedure in the case where $d = 3$ (in the non-twisted setting). Figure~\ref{figure_3Sigma} shows how the coordinate ring 
gradually shrinks upon removal of
$(0,3)$, then of $(3,0)$, and finally of $(0,0)$. The left column shows the graded parts of the original coordinate ring \eqref{gradedcoordinateringdsigma} in degrees $0,1,2,3$, 
while the right column does the same for the eventual graded module described in
\eqref{gradedcoordinateringdsigmaremoved}.
\begin{figure}[thb]
\centering
\begin{minipage}{.23\textwidth}
  \centering
  \includegraphics[width=.18\linewidth]{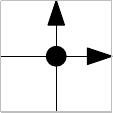}
\end{minipage}
\begin{minipage}{.23\textwidth}
  \centering
  \includegraphics[width=.18\linewidth]{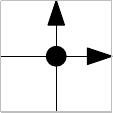}
\end{minipage}
\begin{minipage}{.23\textwidth}
  \centering
  \includegraphics[width=.18\linewidth]{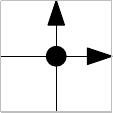}
\end{minipage}
\begin{minipage}{.23\textwidth}
  \centering
  \includegraphics[width=.18\linewidth]{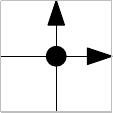}
\end{minipage}
 \\ \vspace{3mm}
\begin{minipage}{.23\textwidth}
  \centering
  \includegraphics[width=.45\linewidth]{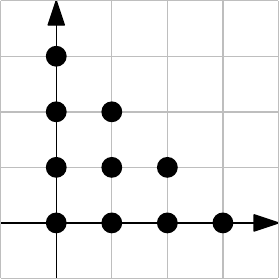}
\end{minipage}
\begin{minipage}{.23\textwidth}
  \centering
  \includegraphics[width=.45\linewidth]{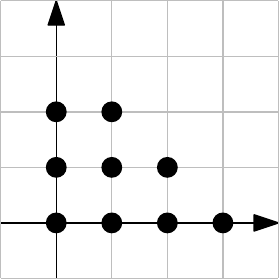}
\end{minipage}
\begin{minipage}{.23\textwidth}
  \centering
  \includegraphics[width=.45\linewidth]{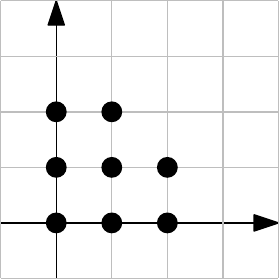}
\end{minipage}
\begin{minipage}{.23\textwidth}
  \centering
  \includegraphics[width=.45\linewidth]{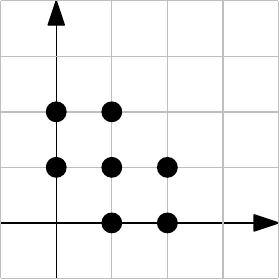}
\end{minipage} \\ \vspace{3mm}
\begin{minipage}{.23\textwidth}
  \centering
  \includegraphics[width=.72\linewidth]{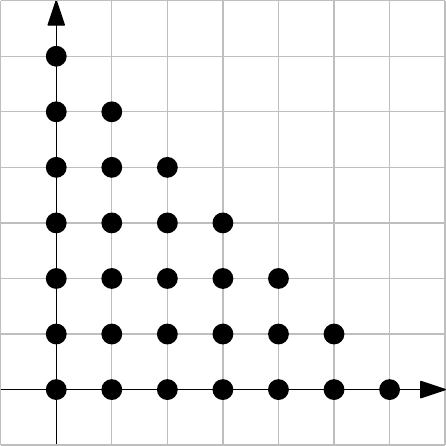}
\end{minipage}
\begin{minipage}{.23\textwidth}
  \centering
  \includegraphics[width=.72\linewidth]{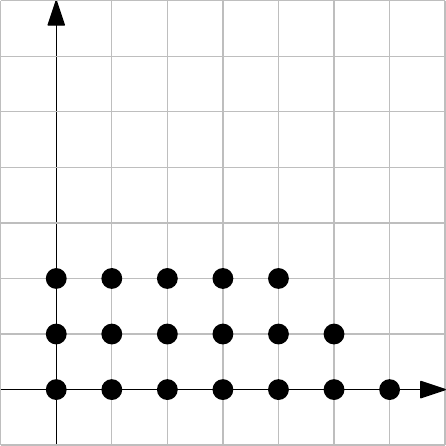}
\end{minipage}
\begin{minipage}{.23\textwidth}
  \centering
  \includegraphics[width=.72\linewidth]{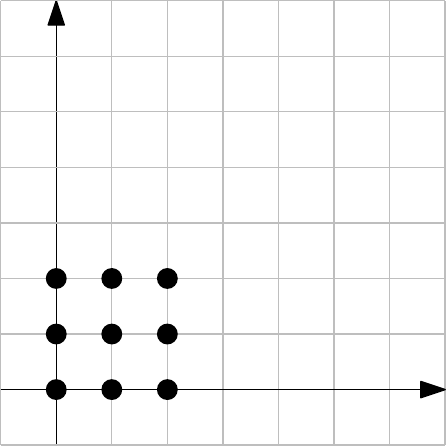}
\end{minipage} 
\begin{minipage}{.23\textwidth}
  \centering
  \includegraphics[width=.72\linewidth]{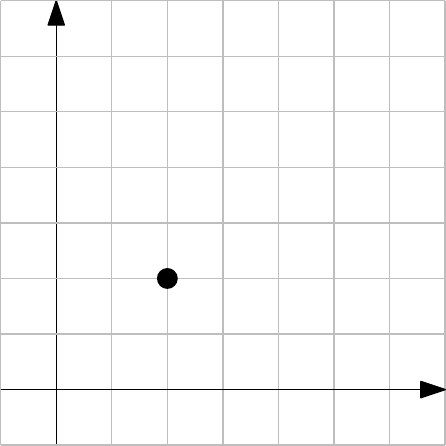}
\end{minipage} \\ \vspace{3mm}
\begin{minipage}{.23\textwidth}
  \centering
  \includegraphics[width=.99\linewidth]{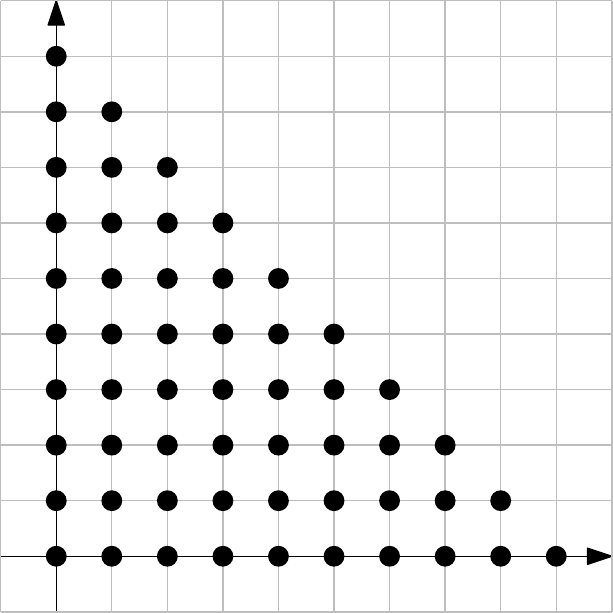}
\end{minipage}
\begin{minipage}{.23\textwidth}
  \centering
  \includegraphics[width=.99\linewidth]{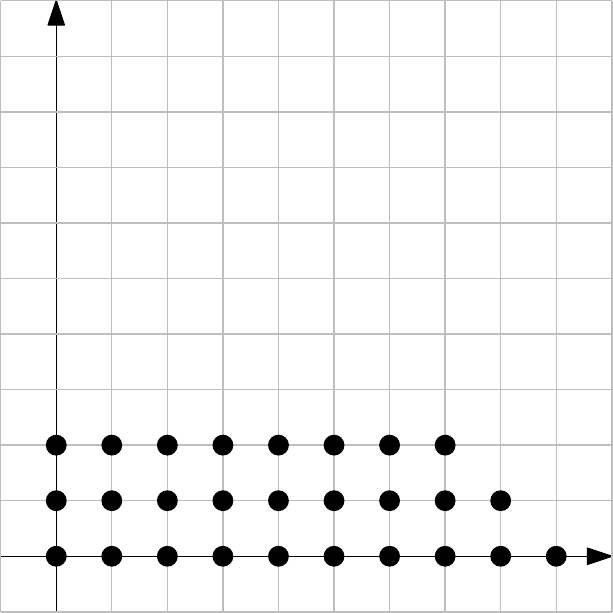}
\end{minipage}
\begin{minipage}{.23\textwidth}
  \centering
  \includegraphics[width=.99\linewidth]{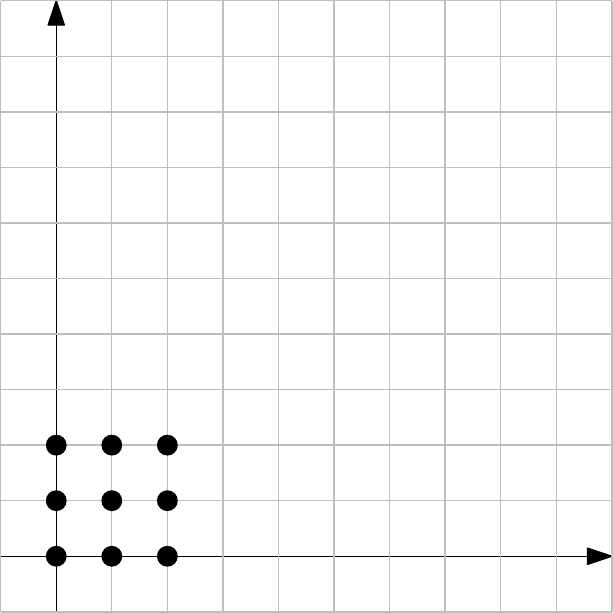}
\end{minipage}
\begin{minipage}{.23\textwidth}
  \centering
  \includegraphics[width=.99\linewidth]{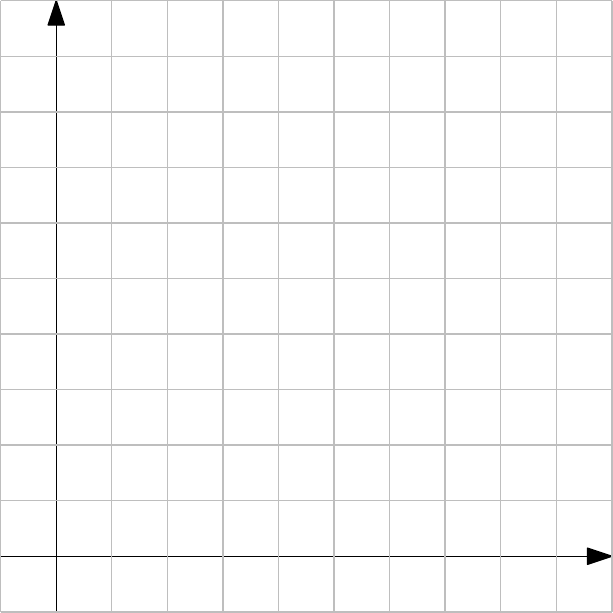}
\end{minipage}
\caption{Removing three points for $\Delta=3\Sigma$}
\label{figure_3Sigma} 
\end{figure}

\section{Computing graded Betti numbers} \label{section_computing}

\subsection{The algorithm}

To compute the entries $b_\ell$ and $c_\ell$
of the graded Betti table \eqref{toricbetti} of $X_\Delta \subseteq \PP^{N_\Delta - 1}$
we use the formulas \eqref{cohombell} and 
\eqref{cohomcelldual}.
In other words, we determine the $b_\ell$'s as  
\[
    \dim \ker \left(     \wedgepow{\ell} V_\Delta \otimes V_{\Delta} 
    \rightarrow
    \wedgepow{\ell-1} V_\Delta \otimes V_{2\Delta} \right) - 
    \dim  \wedgepow{\ell+1} V_\Delta
,\]
while the $c_\ell$'s are computed as
\[
    \dim \ker \left( \wedgepow{\ell-1} V_\Delta \otimes V_{\Delta^{(1)}} 
    \rightarrow
    \wedgepow{\ell-2} V_\Delta \otimes V_{(2\Delta)^{(1)}} \right)
.\]
Essentially, this requires writing down a matrix of the respective linear map and computing its rank.
As explained in Section~\ref{section_bigrading} we can consider these expressions
for each bidegree $(a,b)$ independently, and then just sum the contributions $c_{\ell,(a,b)}^\vee$ resp.\ $b_{\ell,(a,b)}$.
This greatly reduces the dimensions of the vector spaces and hence of the matrices that we need to deal with. 

\begin{remark}
The subtracted term in the formula for $b_\ell$ can be made explicit:
\[ \dim  \wedgepow{\ell+1} V_\Delta = { N_\Delta \choose \ell + 1}. \]
However we prefer to compute its contribution in each bidegree separately
(which is easily done, see Section~\ref{computingdimensions}),
the reason being that the $b_{\ell, (a,b)}$'s are interesting in their own right;
see also Remark~\ref{remark_each_bidegree} below.
\end{remark}

\subsubsection*{Speed-ups}

Lemma~\ref{diagonalterms} allows us to obtain $b_{N_\Delta - 1 - \ell}$ from $c_\ell$
and $c_{N_\Delta - 1 - \ell}$ from $b_\ell$, so 
we only compute one of both. In practice we make an educated guess for what we think will be the easiest option,
based on the dimensions of the spaces involved. Moreover, using Hering and Schenck's Theorem~\ref{heringschenckthm} we find that
$c_\ell$ vanishes as soon as $\ell \geq N_\Delta + 1 - | \partial \Delta \cap \ZZ^2 |$. For this reason
the computation of $b_1, \ldots, b_{| \partial \Delta \cap \ZZ^2 | - 2}$ can be omitted,
which is particularly interesting in the case of the Veronese polygons $d\Sigma$, which have many lattice points on the boundary.

\begin{remark} \label{remark_each_bidegree} 
From the proof of Lemma~\ref{diagonalterms} 
we can extract the formula
\begin{equation} \label{diagonaltermsgraded}
 b_{\ell,(a,b)} - c_{N_\Delta - 1 - \ell, (a,b)} =
 \sum_{j=0}^{\ell+1}(-1)^{j+1}\dim\left(\wedgepow{\ell+1-j} V_\Delta \otimes V_{j\Delta} \right)_{(a,b)}
\end{equation}
for each bidegree $(a,b) \in \mathbb{Z}^2$ and each $\ell = 1, \dots, N_\Delta - 2$. 
Here the subscript on the right hand side indicates
that we consider the subspace of elements having bidegree $(a,b)$.
As explained in Section~\ref{computingdimensions},
we can easily compute the dimensions of the spaces on the right hand side in practice.
Together with \eqref{dualgradingformulas} this allows one to obtain the bigraded parts of the entire Betti table, using essentially the same method.
As an illustration, bigraded versions of some of the data gathered in Appendix~\ref{appendix_data} 
have been made available on \url{http://sage.ugent.be/www/jdemeyer/betti/}.
\end{remark}

We use the material from Section~\ref{section_alexandertruck} to 
reduce the dimensions further. As soon as we are dealing with an $n$-gon with $n \geq 5$, then
by Theorem~\ref{thm-remove} we can remove one lattice point only. In the case of a quadrilateral 
we can remove two opposite vertices. In the case of a triangle we can remove its three vertices. 
For simple computations we just make a random amenable choice.
For larger computations it makes sense to spend a little time on
optimizing the point(s) to be removed, by computing the dimensions of the resulting quotient spaces.

\begin{remark}
As we have mentioned before, from a practical point of view the effect of removing lattice points is somewhat unpredictable. In certain cases we even observed that, although the resulting
matrices are of considerably lower dimension, computing the rank takes more time. We currently have no explanation for this.
\end{remark}

Another useful optimization is to take into account symmetries of $\Delta$,
which naturally induce symmetries of multiples of $\Delta$ and $\Delta\interior$.
For example for $b_\ell$, consider a symmetry $\psi \in \AGL_2(\ZZ)$ of $(\ell+1) \Delta$
and let $(a,b)$ be a bidegree.
Then $b_{\ell, (a,b)} = b_{\ell, \psi(a,b)}$.
The analogous remark holds for $c_\ell$, using symmetries $\psi$ of $(\ell-1) \Delta+\Delta^{(1)}$.

A final speed-up comes from computing in finite characteristic, thereby avoiding inflation of coefficients
when doing rank computations. We believe that this does not affect the outcome, even when computing modulo very small primes such as $2$, but we
have no proof of this fact. Therefore this speed-up comes at the cost of ending up with conjectural graded Betti tables. However recall
from Remark~\ref{remark_semicontinuity} that the graded Betti numbers can never decrease, so the zero entries are rigorous (and 
because of Lemma~\ref{diagonalterms}
the other entry on the corresponding antidiagonal is rigorous as well).

\subsubsection*{Writing down the matrices}

The maps we need to deal with are of the form
\begin{equation}\label{compute-map}
    \wedgepow{p} V_A \otimes V_B
    \stackrel{\delta}{\longrightarrow}
    \wedgepow{p-1} V_A \otimes V_C,
\end{equation}
where $A$, $B$ and $C$ are finite sets of lattice points and $\delta$ is as in \eqref{delta},
subject to the additional rule mentioned in Remark~\ref{additionalrule}.
For a given bidegree $(a,b)$, as a basis of the left hand side of \eqref{compute-map} 
we make the obvious choice
\begin{multline*}
    \{ \, x^{i_1} y^{j_1} \wedge \ldots \wedge x^{i_p} y^{j_p} \otimes x^{i'} y^{j'} \, | \, (i',j') = (a,b) - (i_1, j_1) - \ldots - (i_p, j_p) \text{ and }  \\
        \{(i_1, j_1), \ldots, (i_p, j_p)\} \subseteq A \text{ and } (i',j') \in B \},
\end{multline*}
where $\{(i_1, j_1), \ldots, (i_p, j_p)\}$ runs over all $p$-element subsets of $A$.
In the implementation, we equip $A$ with a total order $<$ and take subsets
such that $(i_1, j_1) < \ldots < (i_p, j_p)$.
We do not need to store the part $x^{i'} y^{j'}$ since that
is completely determined by the rest (for a fixed bidegree).
We use the analogous basis for the right hand side of \eqref{compute-map}.
We then compute the transformation matrix corresponding to the map $\delta$
in a given bidegree, and determine its rank.

Note that the resulting matrix is very sparse: it has at most $p$ non-zero entries in every column,
while the non-zero entries are $1$ or $-1$.
Therefore we use a sparse data structure to store this matrix.

\subsubsection*{Implementation}

We have implemented all this in Python and Cython,
using SageMath~\cite{sagemath} with LinBox~\cite{linbox} for the linear algebra.
In principle the algorithm should 
work equally fine in characteristic zero (at the cost of some efficiency) but for technical reasons
our current implementation does not support this.
For the implementation details we refer to the programming code,
which is made available at \url{https://github.com/jdemeyer/toricbetti}.

\subsection{Computing the dimensions of the spaces}\label{computingdimensions}
Given finite subsets $A, B \subseteq \ZZ^2$, computing the dimension of the space $\wedgepow{p} V_A \otimes V_B$ in each bidegree
can be done efficiently without explicitly constructing a basis.
These dimensions determine the sizes of the matrices involved.
Knowing this size allows to estimate the amount of time and memory
needed to compute the rank. We use this to decide whether to compute $b_\ell$ or $c_{N_\Delta - 1 - \ell}$,
and which point(s) we remove when applying the material from Section~\ref{section_alexandertruck}.

Namely, consider the generating function (which is actually a polynomial)
\begin{equation}
    f_{A}(X,Y,T) = \prod_{(i,j) \in A}(1 + X^i Y^j T).
\end{equation}
Then the coefficient of $X^a Y^b T^p$ is the dimension of the component
in bidegree $(a,b)$ of $\wedgepow{p} V_A$.
The generating function for $\wedgepow{p} V_A \otimes V_B$ then becomes
\begin{equation}
    f_{A,B}(X,Y,T) = \prod_{(i,j) \in A}(1 + X^i Y^j T) \cdot \sum_{(i,j) \in B} X^i Y^j.
\end{equation}
If we are only interested in a fixed $p$, we can compute modulo $T^{p+1}$,
throwing away all higher-order terms in $T$.

\subsection{Applications} \label{section_application}

As a first application we have verified Conjecture~\ref{Kp1conjecture}
for all lattice polygons containing at most $32$ lattice points with at least one lattice point in the interior 
(namely we used the list of polygons from \cite{movingout} and took those polygons for which $N_\Delta \leq 32$).
There are $583\,095$ such polygons; the maximal lattice width that occurs is $8$.
Apart from the ten exceptional polygons
$3 \Sigma, \ldots, 6 \Sigma, \Upsilon_2, \ldots, \Upsilon_6$ and $2 \Upsilon$,
we verified that the entry $b_{N_\Delta - \lw(\Delta) - 1}$ indeed equals zero. In the exceptional
cases, whose graded Betti tables are gathered in Appendix~\ref{appendix_data}, we found that $b_{N_\Delta - \lw(\Delta)}$ equals zero.
Together with Theorem~\ref{Kp1conjectureoneinequality} this proves that Conjecture~\ref{Kp1conjecture} is satisfied for each of these lattice polygons.
The computation was carried out modulo $40\,009$ and
took $1006$ CPU core-days on an Intel Xeon E5-2680~v3.

As a second application we have computed the graded Betti table of the $6$-fold Veronese surface $X_{6\Sigma}$, which can be found in Appendix~\ref{appendix_data}.
Currently the computation was done in finite characteristic only (again $40\,009$) and therefore some of the non-zero entries are conjectural.
The computation took $12$ CPU core-days on an IBM POWER8.
This new data leads to the guesses stated in Conjecture~\ref{veronese_conjecture}, predicting certain entries
of the graded Betti table of $X_{d\Sigma} = \nu_d(\PP^2)$ for arbitrary $d \geq 2$.
\begin{itemize}
\item The first guess states that the last non-zero entry on the row $q = 1$ is given by $d^3(d^2-1)/8$. This is true for $d=2, 3,4,5$ and has been verified in characteristic
$40\,009$ for $d = 6, 7$. 
\item The second guess is about the first non-zero entry 
on the row $q=2$, which we believe to be
\[ { N_{(d\Sigma)\interior} + 8 \choose 9 }.\]
 Here we have less supporting data: it is true for $d=3,4,5$ and has been verified in characteristic $40\,009$ for $d=6$. On the other hand our guess naturally fits 
 within the more widely applicable formula
\[ {  N_{\Delta^{(1)}} - 1 + \left| \{ \, v \in \ZZ^2 \setminus \{ (0,0) \} \, | \, \Delta^{(1)} + v \subseteq \Delta \,  \} \right| \choose N_{\Delta^{(1)}} - 1 }, \]
which we have verified for a large number of small polygons. 
It was discovered and proven to be a lower bound by the fourth author, in the framework
of his Ph.D.\ research; we refer to his upcoming thesis for a proof.
\end{itemize}

\appendix
\section{Some explicit graded Betti tables} \label{appendix_data}

This appendix contains the graded Betti tables of $X_\Delta \subseteq \PP^{N_\Delta - 1}$ for the instances of $\Delta$
that are the most relevant to this paper. The largest of these Betti tables were computed
using the algorithm described in Section~\ref{section_computing}. Because these computations were carried out modulo $40\,009$
the resulting tables are conjectural, except for the zero entries and the entries on the corresponding antidiagonal.
The smaller Betti tables have been verified independently in characteristic zero using the Magma intrinsic~\cite{magma}, along the lines of~\cite[\S2]{canonical}.
For the sake of clarity, we have indicated the conjectural entries by an asterisk. The question marks `???' mean that the corresponding entry has not been computed.

\newcommand{\tableheader}[2]{#1 \mbox{\hspace{0.5em}($N_\Delta=#2$)}:}

\begin{flalign*}
& \tableheader{\Sigma}{3}
&
& \tableheader{2 \Sigma}{6}
&
& \tableheader{3 \Sigma}{10}
\\
& \footnotesize\begin{array}{r|c}
  & 0 \\
\hline
0&
1 \\
1&
0 \\
2&
0
\end{array}
&
& \footnotesize\begin{array}{r|cccc}
  & 0 & 1 & 2 & 3 \\
\hline
0&
1 & 0 & 0 & 0 \\
1&
0 & 6 & 8 & 3 \\
2&
0 & 0 & 0 & 0
\end{array}
&
& \footnotesize\begin{array}{r|cccccccc}
  & 0 & 1 & 2 & 3 & 4 & 5 & 6 & 7 \\
\hline
0&
1 & 0 & 0 & 0 & 0 & 0 & 0 & 0 \\
1&
0 & 27 & 105 & 189 & 189 & 105 & 27 & 0 \\
2&
0 & 0 & 0 & 0 & 0 & 0 & 0 & 1
\end{array}
\end{flalign*}

\begin{flalign*}
& \tableheader{4 \Sigma}{15}
\\
& \footnotesize\begin{array}{r|ccccccccccccc}
  & 0 & 1 & 2 & 3 & 4 & 5 & 6 & 7 & 8 & 9 & 10 & 11 & 12 \\
\hline
0&
1 & 0 & 0 & 0 & 0 & 0 & 0 & 0 & 0 & 0 & 0 & 0 & 0 \\
1&
0 & 75 & 536 & 1947 & 4488 & 7095 & 7920 & 6237 & 3344 & 1089 & 120 & 0 & 0 \\
2&
0 & 0 & 0 & 0 & 0 & 0 & 0 & 0 & 0 & 0 & 55 & 24 & 3
\end{array}
&&
\end{flalign*}

\begin{flalign*}
& \tableheader{5 \Sigma}{21}
\\
& \footnotesize\begin{array}{r|cccccccccl}
  & 0 & 1 & 2 & 3 & 4 & 5 & 6 & 7 & 8 & \cdots \\
\hline
0&
1 & 0 & 0 & 0 & 0 & 0 & 0 & 0 & 0 \\
1&
0 & 165 & 1830 & 10710 & 41616 & 117300 & 250920 & 417690 & 548080 & \cdots \\
2&
0 & 0 & 0 & 0 & 0 & 0 & 0 & 0 & 0
\end{array}
\\
& \footnotesize\begin{array}{r|lcccccccccc}
  & \cdots & 9 & 10 & 11 & 12 & 13 & 14 & 15 & 16 & 17 & 18 \\
\hline
0&&
0 & 0 & 0 & 0 & 0 & 0 & 0 & 0 & 0 & 0 \\
1&\cdots&
568854 & 464100 & 291720 & 134640 & 39780 & 4858 & 375 & 0 & 0 & 0 \\
2&&
0 & 0 & 0 & 0 & 2002 & 4200 & 2160 & 595 & 90 & 6
\end{array}
&&
\end{flalign*}

\begin{flalign*}
& \tableheader{6 \Sigma}{28}
\\
& \footnotesize\begin{array}{r|cccccccccl}
  & 0 & 1 & 2 & 3 & 4 & 5 & 6 & 7 & 8 & \cdots \\
\hline
0&
1 & 0 & 0 & 0 & 0 & 0 & 0 & 0 & 0 \\
1&
0 & 315 & 4950 &  41850 & 240120 & 1024650 & 3415500 & 9164925 & 20189400 & \cdots \\
2&
0 & 0 & 0 & 0 & 0 & 0 & 0 & 0 & 0
\end{array}
\\
& \footnotesize\begin{array}{r|lcccccccl}
  & \cdots & 9 & 10 & 11 & 12 & 13 & 14 & 15 & \cdots \\
\hline
0&&
0 & 0 & 0 & 0 & 0 & 0 & 0 \\
1&\cdots&
36989865 & 56831850 & 73547100 & 80233200 & 73547100 & 56163240 & 35102025 & \cdots \\
2&&
0 & 0 & 0 & 0 & 0 & 0 & 0
\end{array}
\\
& \footnotesize\begin{array}{r|lcccccccccc}
  & \cdots & 16 & 17 & 18 & 19 & 20 & 21 & 22 & 23 & 24 & 25 \\
\hline
0&&
0 & 0 & 0 & 0 & 0 & 0 & 0 & 0 & 0 & 0 \\
1&\cdots&
17305200 & 6177545^\ast & 1256310^\ast & 160398^\ast & 17890^\ast & 945^\ast & 0 & 0 & 0 & 0 \\
2&&
48620^\ast & 231660^\ast & 593028^\ast & 473290^\ast & 218295^\ast & 69300 & 15525 & 2376 & 225 & 10
\end{array}
&&
\end{flalign*}

\begin{flalign*}
& \tableheader{7 \Sigma}{36}
\\
& \footnotesize\begin{array}{r|lcccccccc}
  & \cdots & 26 & 27 & 28 & 29 & 30 & 31 & 32 & 33 \\
\hline
0&&
0 & 0 & 0 & 0 & 0 & 0 & 0 & 0 \\
1&\cdots&
??? & 53352^\ast & 2058^\ast & 0 & 0 & 0 & 0 & 0 \\
2&&
27821664^\ast & 8824410^\ast & 2215136 & 434280 & 64449 & 6832 & 462 & 15
\end{array}
&&
\end{flalign*}

\begin{flalign*}
& \tableheader{\Upsilon = \Upsilon_1}{4}
&
& \tableheader{2 \Upsilon}{10}
\\
& \footnotesize\begin{array}{r|cc}
  & 0 & 1 \\
\hline
0&
1 & 0 \\
1&
0 & 0 \\
2&
0 & 1
\end{array}
&
& \footnotesize\begin{array}{r|cccccccc}
  & 0 & 1 & 2 & 3 & 4 & 5 & 6 & 7 \\
\hline
0&
1 & 0 & 0 & 0 & 0 & 0 & 0 & 0 \\
1&
0 & 24 & 84 & 126 & 84 & 20 & 0 & 0 \\
2&
0 & 0 & 0 & 0 & 20 & 36 & 21 & 4
\end{array}
&&
\end{flalign*}

\begin{flalign*}
& \tableheader{\Upsilon_2}{7}
&
& \tableheader{\Upsilon_3}{11}
\\
& \footnotesize\begin{array}{r|ccccc}
  & 0 & 1 & 2 & 3 & 4 \\
\hline
0&
1 & 0 & 0 & 0 & 0 \\
1&
0 & 7 & 8 & 3 & 0 \\
2&
0 & 0 & 6 & 8 & 3
\end{array}
&
& \footnotesize\begin{array}{r|ccccccccc}
  & 0 & 1 & 2 & 3 & 4 & 5 & 6 & 7 & 8 \\
\hline
0&
1 & 0 & 0 & 0 & 0 & 0 & 0 & 0 & 0 \\
1&
0 & 30 & 120 & 210 & 189 & 105 & 27 & 0 & 0 \\
2&
0 & 0 & 0 & 21 & 105 & 147 & 105 & 40 & 6
\end{array}
&&
\end{flalign*}

\begin{flalign*}
& \tableheader{\Upsilon_4}{16}
\\
& \footnotesize\begin{array}{r|cccccccccccccc}
  & 0 & 1 & 2 & 3 & 4 & 5 & 6 & 7 & 8 & 9 & 10 & 11 & 12 & 13 \\
\hline
0&
1 & 0 & 0 & 0 & 0 & 0 & 0 & 0 & 0 & 0 & 0 & 0 & 0 & 0 \\
1&
0 & 81 & 598 & 2223 & 5148 & 7920 & 8172 & 6237 & 3344 & 1089 & 120 & 0 & 0 & 0 \\
2&
0 & 0 & 0 & 0 & 55 & 450 & 2376 & 4488 & 4950 & 3630 & 1859 & 612 & 117 & 10 \\
\end{array}
&&
\end{flalign*}

\begin{flalign*}
& \tableheader{\Upsilon_5}{22}
\\
& \footnotesize\begin{array}{r|ccccccccccl}
  & 0 & 1 & 2 & 3 & 4 & 5 & 6 & 7 & 8 & 9 & \cdots \\
\hline
0&
1 & 0 & 0 & 0 & 0 & 0 & 0 & 0 & 0 & 0 \\
1&
0 & 175 & 1995 & 11970 & 47481 & 135660 & 290820^\ast & 476385^\ast & 597415^\ast & 581724^\ast & \cdots \\
2&
0 & 0 & 0 & 0 & 0 & 120^\ast & 1575^\ast & 9555^\ast & 52650^\ast & 172172^\ast
\end{array}
\\
& \footnotesize\begin{array}{r|lcccccccccc}
  & \cdots & 10 & 11 & 12 & 13 & 14 & 15 & 16 & 17 & 18 & 19 \\
\hline
0&&
0 & 0 & 0 & 0 & 0 & 0 & 0 & 0 & 0 & 0 \\
1&\cdots&
466102^\ast & 291720^\ast & 134640^\ast & 39780^\ast & 4858^\ast & 375^\ast & 0 & 0 & 0 & 0 \\
2&&
291720^\ast & 338130^\ast & 291720^\ast & 194782^\ast & 102120^\ast & 39900 & 11305 & 2205 & 266 & 15
\end{array}
&&
\end{flalign*}

\begin{flalign*}
& \tableheader{\Upsilon_6}{29}
\\
& \footnotesize\begin{array}{r|lccccccccc}
  & \cdots & 18 & 19 & 20 & 21 & 22 & 23 & 24 & 25 & 26 \\
\hline
0&&
0 & 0 & 0 & 0 & 0 & 0 & 0 & 0 & 0 \\
1&\cdots&
??? & 160398^\ast & 17890^\ast & 945^\ast & 0 & 0 & 0 & 0 & 0 \\
2&&
16095603^\ast & 7911490^\ast & 3140445^\ast & 995280 & 246675 & 46176 & 6150 & 520 & 21
\end{array}
&&
\end{flalign*}

\footnotesize

\vspace{5mm}
\noindent \textsc{Laboratoire Paul Painlev\'e, Universit\'e de Lille-1}\\
\noindent \textsc{Cit\'e Scientifique, 59655 Villeneuve d'Ascq Cedex, France}\\
\vspace{-0.4cm}

\noindent \textsc{Departement Elektrotechniek, KU Leuven and imec}\\
\noindent \textsc{Kasteelpark Arenberg 10/2452, 3001 Leuven, Belgium}\\
\vspace{-0.4cm}

\noindent \emph{E-mail address:} \href{mailto:wouter.castryck@gmail.com}{wouter.castryck@gmail.com}\\

\noindent \textsc{Departement Wiskunde, KU Leuven}\\
\noindent \textsc{Celestijnenlaan 200B, 3001 Leuven, Belgium}\\
\vspace{-0.4cm}

\noindent \emph{E-mail address:} \href{mailto:filip.cools@wis.kuleuven.be}{filip.cools@wis.kuleuven.be}\\

\noindent \textsc{Vakgroep Wiskunde, Universiteit Gent}\\
\noindent \textsc{Krijgslaan 281, 9000 Gent, Belgium}\\
\vspace{-0.4cm}

\noindent \textsc{Laboratoire de Recherche en Informatique, Universit\'e Paris-Sud}\\
\noindent \textsc{B\^at.\ 650 Ada Lovelace, F-91405 Orsay Cedex, France}\\
\vspace{-0.4cm}

\noindent \emph{E-mail address:} \href{mailto:jdemeyer@cage.ugent.be}{jdemeyer@cage.ugent.be}\\

\noindent \textsc{Departement Wiskunde, KU Leuven}\\
\noindent \textsc{Celestijnenlaan 200B, 3001 Leuven, Belgium}\\
\vspace{-0.4cm}

\noindent \emph{E-mail address:} \href{mailto:alexander.lemmens@wis.kuleuven.be}{alexander.lemmens@wis.kuleuven.be}\\

\end{document}